\documentclass[9pt,a4paper, leqno]{amsart}

\usepackage{amssymb,amsmath,amsfonts,amsthm,a4}
\usepackage{graphicx, amsbsy, mathrsfs, esint, bbm, dsfont, bm}
\usepackage{color}
\usepackage{mathtools}
\usepackage{cite}
\usepackage{accents}

\numberwithin{equation}{section}

\newcommand*{\dt}[1]{%
  \accentset{\mbox{\large\bfseries .}}{#1}}
  
\newtheorem{theorem}{Theorem}
\newtheorem{lemma}{Lemma}[section]

\newtheorem{proposition}{Proposition}[section]
\newtheorem{assumption}{Assumption}

\newtheorem*{assumption*}{Assumption}

\newtheorem*{remarks*}{Remarks}
\newtheorem*{remark*}{Remark}

\newcommand{\R}{\mathbb{R}} % reals
\newcommand{\C}{\mathbb{C}} % complex numbers
\newcommand{\N}{\mathbb{N}} % naturals
\newcommand{\T}{\mathbb{T}}

\newcommand{\DD}{\Delta}

\newcommand{\FF}{\mathcal{F}}

\newcommand{\be}{\begin{equation}}
\newcommand{\ee}{\end{equation}}
\newcommand{\bes}{\begin{equation*}}
\newcommand{\ees}{\end{equation*}}

\newcommand{\pt}{\partial}

\catcode`@=11
\def\section{\@startsection{section}{1}%
  \z@{1.5\linespacing\@plus\linespacing}{.5\linespacing}%
  {\normalfont\bfseries\large\centering}}
\catcode`@=12

\renewcommand{\theta}{\vartheta}
\newcommand{\wf}{\widehat{f}}
\newcommand{\wcf}{\widehat{\overline{f}}}
\newcommand{\wg}{\widehat{g}}

\newcommand{\wv}{\widehat{v}}

\renewcommand{\leq}{\leqslant}
\renewcommand{\geq}{\geqslant}
\newcommand{\bsharp}{{\sharp}}

\newcommand{\Copt}{\mathcal{C}}
\newcommand{\weakto}{\rightharpoonup}

\newcommand{\eps}{\varepsilon}
\begin{document}

\title[Rearrangement in Fourier Space and Symmetry Results for PDEs]{A sharp rearrangement principle in Fourier space and symmetry results for PDEs with arbitrary order}

\begin{abstract}
We prove sharp inequalities for the symmetric-decreasing rearrangement in Fourier space of functions in $\R^d$. Our main result can be applied to a general class of (pseudo-)differential operators in $\R^d$ of arbitrary order with radial Fourier multipliers. For example, we can take any positive power of the Laplacian $(-\DD)^s$ with $s> 0$ and, in particular, any polyharmonic operator $(-\DD)^m$ with integer $m \geq 1$.

As applications, we prove radial symmetry and real-valuedness (up to trivial symmetries) of optimizers for: i) Gagliardo--Nirenberg inequalities with derivatives of arbitrary order, ii) ground states for bi- and polyharmonic NLS, and iii) Adams--Moser--Trudinger type inequalities for $H^{d/2}(\R^d)$ in any dimension $d \geq 1$. As a technical key result, we solve a phase retrieval problem for the Fourier transform in $\R^d$. To achieve this, we classify the case of equality in the corresponding Hardy--Littlewood majorant problem  for the Fourier transform in $\R^d$.
\end{abstract}

\author{Enno Lenzmann}
\address{University of Basel, Department of Mathematics and Computer Science, Spiegelgasse 1, CH-4051 Basel, Switzerland.}
\email{enno.lenzmann@unibas.ch}

\author{J\'er\'emy Sok}
\address{University of Basel, Department of Mathematics and Computer Science, Spiegelgasse 1, CH-4051 Basel, Switzerland.}
\email{jeremyvithya.sok@unibas.ch}

\maketitle

\section{Introduction and Main Result}

In the calculus of variations, one of the canonical problems is to study symmetries of optimizers. In particular, it is of interest to understand to what extent optimizers share (or break) symmetries from the underlying variational problem. 

As a concrete example, consider a functional $E : X \to \R$ defined on some Banach space $X$ of complex-valued functions $u : \R^d \to \C$. In many cases, the functional $E[u]$ is invariant under spatial rotations and shifts of the complex phase, which means that
$$
E  [e^{i \alpha} u(R \cdot) ] = E[u] \quad \mbox{for all $R \in O(d)$ and $\alpha \in  \R$}.
$$
A natural question is whether minimizers $Q \in X$ of $E[u]$ (possibly under some additional constraint) must also share this invariance property from the functional $E[u]$ up to trivial symmetries. That is, any minimizer $Q: \R^d \to \R$ is radially symmetric and real-valued, after replacing $Q(x) \to e^{i \alpha} Q(x-x_0)$ with some constants $x_0 \in \R^d$ and $\alpha \in \R$ if necessary.

For a broad class of variational problems where $E[u]$ contains {\em first-order derivatives}, we recall that proving radial symmetry and real-valuedness of optimizers can be typically deduced from a well-established triad of arguments (see e.\,g.~\cite{Sp-74,Hi-76,Ta-76,BrZi-88,AlLi-89,GiNiNi-81}) as follows.
\begin{enumerate}
\item[(A)] The {\em Polya--Szeg\"o inequality} $\| \nabla u^* \|_{L^p} \leq \| \nabla u \|_{L^p}$, where $u^*$ denotes the symmetric-decreasing rearrangement (Schwarz rearrangement) of $u \in W^{1,p}(\R^d)$.
\item[(B)] The {\em moving plane method} for the corresponding Euler-Lagrange equation.
\item[(C)] The elementary inequality $\| \nabla|u| \|_{L^p} \leq \| \nabla u \|_{L^p}$.
\end{enumerate}
However, if we treat variational problems that contain {\em higher-order differential operators}, the methods mentioned above  cannot be applied to deduce symmetry results for optimizers (except for some special cases). Indeed, a model example arises if we wish to minimize the functional
$$
E[u] =  \| \Delta u \|_{L^2}^2  - \| u \|_{L^p}^p 
$$
among all complex-valued functions $u \in H^2(\R^d)$ with $\| u \|_{L^2}=1$, say. For this problem, none of the arguments (A)--(C) can be applied to deduce radial symmetry or real-valuedness of minimizers (up to translation and phase).  For instance, even though we have $\Delta u \in L^p(\R^d)$ it may happen that $\Delta u^* \not \in L^p(\R^d)$ or $\Delta |u| \not \in L^p(\R^d)$, showing that both (A) and (C) are not at our disposal. As for (B), that the moving plane method cannot be applied to the corresponding Euler-Lagrange equation
$$
\DD^2 u + \lambda u -|u|^{p-2} u = 0 \quad \mbox{in $\R^d$ with some $\lambda >0$},
$$
due to the lack of a maximum principle for $\DD^2 + \lambda$ in $\R^d$.

In the present paper, we propose a new strategy to improve the situation for differential operators of arbitrary order by considering rearrangements in Fourier space. In particular, we will obtain a useful substitute for the Polya--Szeg\"o rearrangement inequality in $\R^d$ and analyze the case of equality.

To formulate our main result, we will first need to introduce some basic definitions as follows. Given a function $f$ in $L^2(\R^d)$, say, we define its {\bf Fourier rearrangement} to be the function $f^\sharp \in L^2(\R^d)$ given by
\be
f^\sharp = \FF^{-1} \left \{ (\FF f)^{*} \right \} .
\ee  
Here $\FF$ denotes the Fourier transform on $\R^d$ (with its inverse $\FF^{-1}$) and  $g^* : \R^d \to [0, \infty)$ stands for the symmetric-decreasing rearrangement (Schwarz symmetrization) of a measurable function $g : \R^d \to \C$ that vanishes at infinity. See Section \ref{sec:prelim} for more details.

Evidently, the mapping $f \mapsto f^\sharp$ is highly nonlinear (due to the symmetric-rearrangement) and nonlocal (due taking the Fourier transform). Thus, the analytic fine properties of operation may be challenging to study. But some of basic features can be obtained in a straightforward fashion. First of all, the radial symmetry of $(\FF f)^*$ and its real-valuedness together imply that the Fourier rearrangement $f^\sharp : \R^n \to \R$ is always a radially symmetric and real-valued function. Of course, unlike its symmetric-decreasing rearrangement $f^*(x) \geq 0$, the function $f^\sharp(x)$ may change its sign. (But under the extra assumption that $\wf \in L^1(\R^d)$, the function $f^\sharp(x)$ is always {\em positive definite} in the sense of Bochner; see Section \ref{sec:applications} below for a potential relevance of this concerning ground states for polyharmonic NLS.)

A further straightforward and useful property of the Fourier rearrangement is that it does not increase the norm for the Sobolev spaces $H^s(\R^d)$ for all $s \geq 0$. In particular, we readily find the general inequality
\be \label{ineq:polya_fourier}
\| (-\Delta)^s f^\sharp \|_{L^2} \leq \| (-\Delta)^s f \|_{L^2}
\ee
valid for any $f \in H^s(\R^d)$ and all $s \geq 0$. Moreover, it is elementary to verify that the mapping $f \mapsto f^\sharp$ is not just bounded but it is indeed a continuous map on $H^s(\R^d)$ for all $s \geq 0$ (see below). As we will see, this preservation of regularity makes the Fourier rearrangement well suited for the study of symmetry properties of solutions to variational PDEs in $\R^d$ that contain (pseudo-)differential operators of arbitrary order (e.\,g.~the biharmonic operators $\DD^2$).

Let us emphasize that the general inequality \eqref{ineq:polya_fourier} is valid for all $s \geq 0$. Hence, it is in stark contrast to the celebrated Polya--Szeg\"o principle for the symmetric-decreasing rearrangement, which fails in $H^s(\R^d)$ when $s>1$ and therefore has only very limited use when dealing with higher-order derivatives. Furthermore, we recall the remarkable result proved by Almgren and Lieb \cite{AlLi-89}, who showed that  the map $f \mapsto f^*$ fails to be continuous in $H^1(\R^d)$ for $d \geq 2$, although it is a bounded map on $H^1(\R^d)$. Of course, such pathologies cannot exist for the Fourier rearrangement on any of the Sobolev spaces $H^s(\R^d)$ with $s \geq 0$.\footnote{As an amusing aside, we remark that, by Plancherel, the Almgren--Lieb result carries over to the Fourier rearrangement acting in the weighted space $L^2_w=L^2(\R^d, |x|^2 \, dx)$, i.\,e., the map $f \mapsto f^\sharp$ fails to be continuous in $L^2_w$ in dimensions $d \geq 2$, although we have the boundedness property $\| f^\sharp \|_{L^2_w} \leq \| f\|_{L^2_w}$.} 

A canonical question concerning \eqref{ineq:polya_fourier} is to study the case of equality. In view of the classical result by Brothers and Ziemer \cite{BrZi-88} on the case of equality in the Polya--Szeg\"o inequality, it seems natural to ask whether equality in \eqref{ineq:polya_fourier} with $s > 0$ implies that 
\be \label{eq:equal_f}
f(x)= e^{i \alpha} f^\sharp(x-x_0) \quad \mbox{for a.\,e.~$x \in \R^d$}
\ee
with some constants $\alpha \in \R$ and $x_0 \in \R^d$. However, such a claim is too naive and cannot be true in general. In fact, a closer inspection shows that equality in \eqref{ineq:polya_fourier} with $s > 0$ occurs if and only if the Fourier transforms $\wf$ satisfies
\be \label{eq:f_fstar}
|\wf(\xi)| = (\wf)^*(\xi) \quad \mbox{for a.\,e.~$\xi \in \R^d$}.
\ee
Clearly, this type of information alone is far too weak to conclude that a strong conclusion like \eqref{eq:equal_f} holds. Indeed, if we pick an arbitrary measurable phase function $\theta : \R^d \to \R$ and replace 
$$
\wf(\xi) \to e^{i \vartheta(\xi)} \wf(\xi),
$$
 we see that $f$ and $f^\sharp$ cannot be in general related as in \eqref{eq:equal_f} above. 
 
In some sense, the task to recover $f$ from equation \eqref{eq:f_fstar} with some suitable extra conditions can be seen as a {\bf phase retrieval problem} in Fourier analysis. That is, given the modulus $|\wf|$ one tries to determine the  phase function $\theta$ up to some manageable degree of ambiguity. In this direction, a typical condition is to impose that $\wf$ is analytic on $\C^d$ (with some growth condition). Such an assumption, by Paley--Wiener theory, forces the function $f$ to be compactly supported in $\R^d$. However, a restriction to compactly supported $f$ would be of no use in the situation we are interested in here (e.\,g., think of $f$ being an optimizer that solves some PDE, which typically have a unique continuation property and hence $f$ cannot vanish on open non-empty sets in $\R^d$).

To solve the kind of phase retrieval problem mentioned above, we will use a different strategy as follows. In fact, we will prove below that the freedom to choose $\vartheta : \R^d \to \R$ will dramatically diminish, if we additionally impose that equality holds in another type of inequality for the $L^p$-norm of $f$, which naturally arises in variational problems. More precisely, by applying the Brascamp--Lieb--Luttinger multilinear convolution inequality on the Fourier side, we deduce the inequality
\be \label{ineq:Lp_hardy}
\| f \|_{L^{p}} \leq \| f^\sharp \|_{L^p},
\ee
provided that $p > 2$ is an even integer or $p=\infty$. Now, by imposing that $\wf$ is continuous, we can conclude that joint equality in {\em both} inequalities \eqref{ineq:polya_fourier} and \eqref{ineq:Lp_hardy} can occur if and only if the phase function $\theta$ is {\em affine}, i.\,e., we have
$$
\widehat{f}(\xi) = e^{i (\alpha +  x_0 \cdot \xi)} |\widehat{f}|^*(\xi) \quad \mbox{for a.\,e.~$\xi \in \R^d$},
$$
with some constants $\alpha \in \R$ and $x_0 \in \R^d$. From this the desired claim \eqref{eq:equal_f} immediately follows.  We refer to Theorem \ref{thm:main} below for a more general and precise statement. 

As a main technical ingredient in our proof, we will classify the case of equality in the so-called {\bf Hardy--Littlewood majorant problem} in $\R^d$ for the $L^{p}$-norms with $p \in 2 \N \cup \{ \infty \}$.  Here, we will make essential use of the fact that the set $\Omega = \{ |\wf(\xi)| > 0 \}$ is {\em connected} in $\R^d$. In particular, this topological property holds true in our case, since $|\wf|= (\wf)^*$ implies that the set $\Omega$ is either some open ball around the origin or all of $\R^d$. We refer to Section \ref{sec:major} below for more details on how the topology of $\Omega$ enters the proof.

\medskip

We will now formulate our first main result. Since our methods easily generalize to differential operators including any positive powers of the Laplacian $(-\Delta)^s$ in $\R^d$, we  introduce a suitable class of rotationally invariant (pseudo-)differential operators $L$ on $\R^d$ that satisfy the following conditions.

\begin{assumption} \label{ass:L}
Let $d \in \N$ and suppose $\omega : \R^d \to \R$ is a real-valued and locally integrable function. Let $L$ denote the corresponding (pseudo-)differential operator defined as 
$$
\widehat{(L f)}(\xi) = \omega(\xi) \widehat{f}(\xi)
$$ 
for Schwartz functions $f \in \mathcal{S}(\R^d)$. We assume that the following conditions hold.
\begin{itemize}
\item[(i)] There exist constants $s \geq 0$ and $C \geq 0$ such that
$$
0 \leq \omega(\xi) \leq C \left ( 1 + |\xi|^{2s} \right ) \quad \mbox{for a.\,e.~$\xi \in \R^d$}.
$$
\item[(ii)] The function $\omega(\xi)$ is {\em \textbf{radially symmetric}} and {\em \textbf{strictly increasing}} with respect to $|\xi|$, i.\,e., we have
$$
\omega(\xi) = \omega(\eta) \ \ \mbox{if $|\xi| = |\eta|$} \quad \mbox{and} \quad  \omega(\xi) < \omega(\eta) \ \ \mbox{if $|\xi| < |\eta|$}.
$$
\end{itemize}
\end{assumption}

\begin{remark*} {\em  Typical examples for operators $L$ in $\R^d$ are as follows.
\begin{itemize}
\item $L=(-\DD)^s$ and $L=(-\Delta + 1)^{s/2}$  with $s > 0$. 
\item $L =(-\Delta)^{s} \coth (-\Delta)^{s}$ with $s > 0$. For $s=1/2$ and $d=1$, this operator arises in the intermediate long-wave equation (ILW) for water waves.
\item $L$ with multiplier $\omega(\xi)= 1-\sqrt{\tanh (h \xi)/\xi}$ on $\R$ with some constant $h>0$. The operator $L$ arises in the Whitham equation for water waves. Note that $L$ is bounded on $L^2(\R)$ so that Assumption 1 is satisfied with $s=0$.
\end{itemize}
If the operators $L_1, \ldots, L_m$ satisfy Assumption 1, so does any linear combination $L= \sum_{i=1}^m c_i L_i$ with coefficients $c_i \geq 0$. For instance, we can take $L= \DD^2 - \beta \DD$ with $\beta \geq 0$.
}
\end{remark*}

Throughout the following, we let $\langle f, g \rangle = \int_{\R^d} \overline{f(x)} g(x) \, dx$ denote the standard inner product on $L^2(\R^d)$. By Plancherel's identity, we have
$$
\langle f , L f \rangle = \int_{\R^d} \omega(\xi) |\wf(\xi)|^2 \, d \xi,
$$
which is well-defined for $f \in H^s(\R^d)$ provided $L$ satisfies Assumption \ref{ass:L}.  

We now can state our main result, which provides the following strict rearrangement inequalities involving the quantities $\langle f, L f \rangle$ and $\| f \|_{L^p}$ with even integer $p>2$ or $p = \infty$. 

\begin{theorem}[A Strict Rearrangement Principle in Fourier Space] \label{thm:main}
Let $d \in \N$ and suppose $L$ satisfies {\em \textbf{Assumption \ref{ass:L}}} above with some $s \geq 0$. Assume that $p > 2$ is an even integer or $p=\infty$ and let $1 \leq p' \leq 2$ be its conjugate exponent, i.\,e., $1/p+1/p'=1$. 

Then, for any $f \in H^s(\R^d) \cap \FF(L^{p'}(\R^d))$, we have $f^\bsharp \in H^s(\R^d) \cap \FF(L^{p'}(\R^d))$ and
\be \tag{$\spadesuit$} \label{ineq:main}
\langle f^\bsharp, L f^\bsharp \rangle \leq \langle f, L f \rangle \quad \mbox{and} \quad \| f \|_{L^p} \leq \| f^\sharp \|_{L^p}.
\ee

In addition, assume  that $\widehat{f}(\xi)$ is continuous. Then equality occurs in  both inequalities in \eqref{ineq:main} if and only if $f(x)$ equals its Fourier rearrangement $f^\bsharp(x)$ up to a constant phase and translation, i.\,e.,
$$
f(x) = e^{i \alpha} f^\bsharp(x-x_0) \quad \mbox{for a.\,e.~$x \in \R^d$},
$$
with some constants $\alpha \in \R$ and $x_0 \in \R^d$. In particular, the function $f(x)$ is radially symmetric and real-valued (up to translation and constant phase).
\end{theorem} 

\begin{remarks*}{\em 
1) The condition $f \in \FF(L^{p'}(\R^d))$ becomes superfluous if the exponent $p > 2$ is {\em $H^s(\R^d)$-subcritical}, i.\,e.,
\be \label{def:ps} 
2 < p < p_*(d,s) := \begin{dcases*} \frac{2d}{d-2s} & if $0 < s < d/2$, \\ \infty & if $s \geq d/2$. \end{dcases*} 
\ee
Furthermore, we can also take $p=\infty$ if $s > d/2$. In fact, by H\"older's inequality, we find $\| \widehat{f} \|_{L^{p'}} \leq C \| (1+|\xi|^2)^{s/2} \widehat{f} \|_{L^2} \leq C \| f\|_{H^s}$ under the condition on $p$ stated above. 

2) By Plancherel and the equimeasurability property of symmetric-decreasing rearrangements, we have that 
$$
\| f \|_{L^2} = \| f^\sharp \|_{L^2} \quad \mbox{for all $f \in L^2(\R^d)$}.
$$ 
Thus the classification of the equality case in \eqref{ineq:main} must clearly fail for $p=2$. 

3) The continuity assumption for $\wf$ is due to technical convenience and it can probably be relaxed (or completely dropped). However, in the applications of Theorem \ref{thm:main} below, it will turn out there that the continuity of $\widehat{f}$ can be deduced from the corresponding Euler-Lagrange equation satisfied by the optimizers. See below for details.

4) As shown below, the set 
$$
H^{s, \sharp}(\R^d) = \{ f \in H^s(\R^d) : f=f^\sharp \}
$$
is weakly closed in $H^s(\R^d)$ and compactly embedded in $L^p(\R^d)$ for all $2 < p < p_*(s,d)$. This compactness fact can be used to provide a slick  existence proof of minimizers for variational problems that are amenable to Fourier rearrangement; see Section \ref{sec:compact} below.
}
\end{remarks*}

\subsection*{Acknowledgments}
Both authors were supported by the Swiss National Science Foundations (SNF) through Grant No.~20021-169464. They also wish to thank Rupert Frank for drawing our attention to \cite{Montgomery76, Frank2016} and for his helpful correspondence \cite{Frank2018} about the real-valuedness  of optimizers for the Gagliardo--Nirenberg inequality and 
E.\,L.~is grateful to Rowan Killip for pointing out the work in \cite{Bo-62} on the Hardy--Littlewood majorant problem. Finally, the authors wish to thank Maria Ahrend for her careful proofreading.

\section{Applications: Radial Symmetry of Optimizers}

\label{sec:applications}

In this section, we discuss some  applications of Theorem \ref{thm:main} to show radial symmetry of optimizers. More precisely, we will consider the following examples.
\begin{itemize}
\item Gagliardo--Nirenberg type inequalities with derivatives of arbitrary order.
\item Ground states for generalized NLS type (e.\,g.~fourth-order NLS).
\item Adams--Moser--Trudinger type inequalities in $\R^d$.
\end{itemize}
This list of  applications  is by far non-exhaustive and further applications of Theorem \ref{thm:main}  will be addressed in future work.  

\subsection{Gagliardo--Nirenberg Type Inequalities}
Suppose $d \in \N$ and let $s > 0$ be a real number in what follows. We consider the following Gagliardo--Nirenberg (GN) type interpolation inequality:
\be \tag{GN} \label{ineq:GN}
\| f \|_{L^p} \leq \Copt_{d,s,p} \| (-\Delta)^{\frac s 2} f \|_{L^2}^\theta \| f \|_{L^2}^{1-\theta},
\ee
which is valid for any $f \in H^s(\R^d)$, assuming that we take
\be
2 < p < p_*(s,d) \quad \mbox{and} \quad \theta = \frac{d (p-2)}{2sp},
\ee
where we recall the definition  of $p_*(s,d)$ from \eqref{def:ps}. By Plancherel's theorem, the special case of integer $s = m \in \N$ can be written as
\be \label{ineq:GN_m}
\| f \|_{L^p} \leq \Copt_{d,m,p} \| \nabla^m f \|_{L^2}^\theta \| f \|_{L^2}^{1-\theta},
\ee
where we denote  $\nabla^m = \nabla (-\Delta)^{(m-1)/2}$ if $m$ odd and $\nabla^m = (-\Delta)^{m/2}$ if $m$ is even. In particular, the Gagliardo--Nirenberg inequality with $m=1$ is of central importance in the study of NLS; see \cite{We-83}. 

Indeed, the existence of an optimal constant $\Copt_{d,s,p}\in (0, \infty)$ as well as optimizers $Q \in H^s(\R^d)$ for inequality \eqref{ineq:GN} can be deduced from standard variational arguments (e.\,g., concentration-compactness methods, see also \cite{BeFrVi-14} for an approach using a different compactness lemma). In the appendix, we will provide an alternative and rather elementary existence proof of optimizers by using the Fourier rearrangement techniques. 

%In the case $s > 1$, the Polya--Szeg\"o principle or moving plane methods (based on the maximum principle) are not applicable to conclude that optimizers $Q$ must be radially symmetric (up to translations). In fact, even a proof of the fact that $Q$ has to be real-valued (up to a constant phase) has seemed to be out of scope by known techniques. 

A direct application of Theorem \ref{thm:main} yields the following symmetry result for optimizers of inequality \eqref{ineq:GN}, which also covers all $s > 0$,  provided that $p>2$ is an even integer.

\begin{theorem} \label{thm:GN}
Let $d \in \N, s > 0$, and suppose $2 <  p < p_\star$ is an even integer. Then any optimizer $Q \in H^s(\R^d)$ for the Gagliardo--Nirenberg inequality \eqref{ineq:GN} satisfies
\be
Q(x) = e^{i \alpha} Q^\sharp(x-x_0) \quad \mbox{for a.\,e.~$x \in \R^d$},
\ee
with some constants $\alpha \in \R$ and $x_0 \in \R^d$. As a consequence, the function $Q (x)$ is radially symmetric on $\R^d$ (up to translation) and real-valued (up to a constant phase).
\end{theorem}

\begin{remarks*}
{\em
1) In the case $s \leq 1$, the fact that optimizers $Q(x)$ can be chosen real-valued (up to a constant phase) can be deduced from classical inequality
$$
\| (-\DD)^{s/2} |f| \|_{L^2} \leq \| (-\DD)^{s/2} f \|_{L^2},
$$
which is valid only if $s \in [0,1]$. This property can also be seen as the so-called {\em first Deny--Beurling criterion} saying that the corresponding heat semigroup $e^{-t (-\DD)^{s}}$ is {\em positivity preserving} if and only if $s \in [0,1]$ holds; see, e.\,g.,~\cite[Appendix 1 to Section XIII.12]{ReSi-72}. However, for $s > 1$, no such argument can be applied to deduce real-valuedness of $Q$ (up to a constant phase). As an alternative to Fourier rearrangements,  the real-valuedness of optimizers for \eqref{ineq:GN} can also be inferred from arguments in \cite{Frank2016}[Section 6]; see \cite{Frank2018} for an adaptation, which  also applies to  non-even integer $p>2$. 

2) In particular, we obtain radial symmetry of optimizers for \eqref{ineq:GN} in the {\em biharmonic case}, where the corresponding Euler-Lagrange equation for $Q \in H^s(\R^d)$ is given by
\be \label{eq:bi_NLS}
 \Delta^2 Q + Q - |Q|^{p-2} Q = 0 \quad \mbox{in $\R^d$}.
\ee

3) For $s=1$, the radial symmetry of $Q(x)$ (and its uniqueness up to symmetries) is a classical fact based on the moving plane method \cite{GiNiNi-79} and uniqueness of radial ground states for $s=1$ (see \cite{Kw-89}). We refer to \cite{FrLe-13, FrLeSi-16} for radial symmetry and uniqueness of optimizers in the fractional case $0 < s < 1$. However, the methods for the range $0 < s \leq 1$ do not have known generalization to $s > 1$.

4) We can show that $\widehat{Q} \in L^1(\R^d)$ by using the Euler-Lagrange equation. Thus we see that $Q^\sharp$ is (up to translation and phase) always a {\em positive definite} function in the sense of Bochner. That is, for any choice of points $x_1, \ldots, x_N \in \R^d$, the $N \times N$-matrix $(Q(x_k-x_l))_{1 \leq k,l \leq N}$ is positive semi-definite. In particular, it holds that
\be \label{ineq:boch}
Q(0) \geq |Q(x)| \quad \mbox{for all $x \in \R^d$}.
\ee
In the biharmonic case $s=2$, numerical work in \cite{FiIlPa-02} indicates that the corresponding Euler-Lagrange equation \eqref{eq:bi_NLS} has (at least) {\em two} different radially symmetric real-valued solutions $Q \not \equiv 0$ in $H^2(\R^d)$. But only one of them satisfies condition \eqref{ineq:boch}. We believe that the property of positive definiteness (and perhaps further properties of $\widehat{Q} \geq 0$ such as showing that $\xi \mapsto \widehat{Q}(\xi)$ is log-concave) may be play a central role when proving {\em uniqueness} of optimizers for the general Gagliardo--Nirenberg inequality \eqref{ineq:GN}.
}

\end{remarks*}

\begin{proof}[Proof of Theorem \ref{thm:GN}]
Let $Q \in H^s(\R^d)$ be an optimizer for \eqref{ineq:GN}. Since $p < p_\star$ is $H^s$-subcritical, an application of H\"older's inequality in Fourier space shows that $\widehat{Q} \in L^{p'}(\R^d)$ with $1/p+1/p'=1$. Thus we can apply Theorem \ref{thm:main} to deduce that the Fourier rearrangement $Q^\sharp$ is an optimizer as well and we must have the equalities $\| (-\Delta)^{\frac s 2} Q^\sharp \|_{L^2}  = \| (-\Delta)^{\frac s 2} Q\|_{L^2}$ and $\| Q ^\sharp \|_{L^p} = \| Q \|_{L^p}$. To conclude the proof of Theorem \ref{thm:GN}, it remains to show that $\widehat{Q}$ is continuous. Indeed, after rescaling $Q(x) \to a Q(bx)$ with some constants $a, b > 0$, we find that $Q$ solves the corresponding Euler--Lagrange equation given by
\be \label{eq:EL_GNS}
(-\Delta)^s Q + Q - |Q|^{p-2} Q = 0 \quad \mbox{in $\R^d$.}
\ee
Since $p > 2$ is an even integer, we obviously have $p \geq 4 \geq 3$. Therefore, it follows that the function $F = |Q|^{p-2} Q \in L^1(\R^d)$ by H\"older's inequality and the fact that $Q \in L^2(\R^d) \cap L^{p}(\R^d)$. Thus the Fourier transform $\widehat{F}=\widehat{( |Q|^{p-2} Q)}$ is continuous and we deduce from the Euler-Lagrange equation in Fourier space that $\widehat{Q}(\xi) = (|2 \pi \xi|^{2s} + 1)^{-1} \widehat{F}(\xi)$ is continuous as well.

Thus we can apply  Theorem \ref{thm:main} to complete the proof of Theorem \ref{thm:GN}.
\end{proof}

\begin{remark*}{\em 
In the case $2 < p < 3$, any solution $Q \in H^s(\R^d)$ of \eqref{eq:EL_GNS} also has the property that $\widehat{Q}$ is continuous. However, the proof would require more sophistication than the case above when $p \geq 3$. A possible way to prove that  $\widehat{Q} \in C^0$ would be to show that $Q \in L^1(\R^d)$, which can be deduced by suitable spatial decay estimates for $Q(x)$. }
\end{remark*}

\subsection{Ground States for NLS Type Equations of Arbitrary Order}

As our next application of Theorem \ref{thm:main}, we discuss radial symmetry of ground state solutions for time-independent NLS type equations with derivative of arbitrary order. For example, the fourth-order equation for $Q \in H^2(\R^d)$ solving
\be \label{eq:fNLS}
\DD^2 Q - \beta \DD Q + \omega Q - |Q|^{p-2} Q = 0 \quad \mbox{in $\R^d$},
\ee
with parameters $\beta \geq 0$ and $\omega > 0$ provides solitary wave solutions for the {\em fourth-order NLS} (see e.\,g.~\cite{BoCaGoJe-18,BoLe-17,FiIlPa-02} and references given there) and the corresponding time-dependent equation arises as a model equation in nonlinear optics. The case of vanishing $\beta=0$ is referred to as the {\em biharmonic NLS}. Due to the general lack of a maximum principle for the biharmonic operator $\DD^2$ in $\R^d$ or Polya--Szeg\"o rearrangement principles in $H^2(\R^d)$, the question of radial symmetry of  {\em ground state} solutions $Q \in H^2(\R^d)$ for \eqref{eq:fNLS} has been mainly open so far; see below for some known partial results in this direction. However, if we assume that $p \in 2\N$, then Theorem \ref{thm:main} will show that any ground state solution $Q$ must be radially symmetric and real-valued (up to translation and constant phase).  In what follows, we will consider a general class of operators including $\DD^2 - \beta \DD$ with $\beta \geq 0$. See, e.\,g., \cite{We-87, We-87b} where NLS type equations with general dispersions where introduced.

Let $d \in \N$, $s> 0$, and $2 < p < p_*(s,d)$. We suppose that $L$ is a (pseudo-)differential operator that satisfies Assumption \ref{ass:L}. For a given real number $\omega > 0$, we define the action functional $A : H^s(\R^d) \to \R$ by setting
\be
A(u) = \frac 1 2 \langle u, L u \rangle + \frac{\omega}{2} \| u \|_{L^2}^2 - \frac{1}{p} \| u \|_{L^{p}}^{p} .
\ee
The corresponding Euler--Lagrange equation $A'(Q)=0$ is easily found to be
\be 
L Q + \omega Q - |Q|^{p-2} Q = 0 \quad \mbox{in $\R^d$.}
\ee
By definition, a ground state $Q \in H^s(\R^d)$, $Q \not \equiv 0$ is a minimizer of $A(u)$ among all its (non-trivial) critical points. That is, the set of ground states for $A(u)$ can be written as
\be
\mathcal{G} = \big \{ u \in H^s(\R^d)\setminus \{ 0 \}: \mbox{$A(u) \leq A(v)$ for all $v \in H^s(\R^d) \setminus \{ 0 \}$ with $A'(v) = 0$}  \big \}.
\ee
It is elementary to check that ground states for $A(u)$ can be equivalently described by considering the minimization problem
\be \label{def:min_m}
\mathfrak{m} = \inf \left \{ \frac{1}{2} \langle v, L v \rangle + \frac{\omega}{2} \| v \|_{L^2}^2 : v \in H^s(\R^d), \; \| v \|_{L^p}=1 \right \} .
\ee
Indeed, it is straightforward to verify that $u \in H^2(\R^d) \setminus \{ 0 \}$ is a ground state for $A(u)$ if and only if $v = u / \| u \|_{L^p}$ is a minimizer for \eqref{def:min_m}.

We have the following symmetry result.

\begin{theorem} \label{thm:gs}
Let $d \in \N$, $s>0$, $\omega > 0$, and suppose $2 < p <  p_*(s,d)$ is an even integer. Let $L$ satisfy Assumption \ref{ass:L} and assume in addition that its symbol $\omega : \R^n \to \R$ is continuous. Then any ground state $Q \in \mathcal{G}$ for the action functional $A(u)$ satisfies
$$
Q(x) = e^{i \alpha} Q^\sharp(x-x_0) \quad \mbox{for a.\,e.~$x \in \R^d$},
$$
with some constants $\alpha \in \R$ and $x_0 \in \R^d$. Consequently, the function $Q(x)$ is radially symmetric and real-valued (up to translation and constant phase).
\end{theorem}

\begin{remarks*}
{\em 1) For $L=(-\Delta)^s$, it is elementary to check that a ground state $Q \in \mathcal{G}$ is also an optimizer for the Gagliardo--Nirenberg inequality (GN). Conversely, any optimizer $Q$ for (GN) is (after a suitable rescaling) also a ground state for $A(u)$ when $L=(-\Delta)^s$.

2) For $L=\DD^2 - \beta \DD$ with $\beta \geq 2\sqrt{\omega}$, is was shown in\cite{BoNa-15} that $Q$ is radially symmetric and positive $Q(x) > 0$ (modulo translation and a constant factor). In this case, the fourth-order equation can be written as a corporative system of two second-order equations so that the maximum-principle type arguments can be applied. Thus, in some sense, the Laplacian part $-\beta \Delta$ ``dominates'' the biharmonic operator $\DD^2$ for large $\beta$ and ground states $Q$ essentially behave as in the second-order case. However, in the pure biharmonic case when $\beta=0$, any non-trivial solution $Q \in H^2(\R^d)$ of \eqref{eq:fNLS} must be sign-changing, which displays a completely different behavior than in the case $\beta \geq 2 \sqrt{\omega}$.}
\end{remarks*}

\begin{proof}
Not surprisingly, the proof follows closely the arguments given for Theorem \ref{thm:GN} above. First, we recall if $Q$ is a ground state then the normalized function
$$
v = \frac{1}{ \| Q \|_{L^p}} Q
$$ 
is a minimizer for problem \eqref{def:min_m}. Next, we define the quadratic functional
$$
T(v) = \frac{1}{2} \langle v, L v\rangle + \frac{\omega}{2} \| v \|_{L^2}^2 \quad \mbox{for $v \in H^s(\R^d)$.}
$$ 
From Theorem \ref{thm:main} together with $\| v^\sharp \|_{L^2} = \| v \|_{L^2}$, we deduce that
$$
T(v^\sharp) \leq T(v) \quad \mbox{and} \quad \| v \|_{L^p} \leq \| v^\sharp \|_{L^p}.
$$
We claim that equality $1=\| v \|_{L^p} = \| v^\sharp \|_{L^p}$ must hold. Suppose this was false and that we had $\alpha = \| v \|_{L^p}/ \| v^\sharp \|_{L^p} < 1$. But then $w = \alpha v^\sharp$ would satisfy $T(w) = \alpha^2 T(v^\sharp) \leq \alpha^2 T(v) < T(v) = \mathfrak{m}$, contradicting $T(w) \geq \mathfrak{m}$ because of $\| w \|_{L^p} = 1$. Thus the equality $\| v \|_{L^p} = \| v^\sharp \|_{L^p}$ must hold and a-posteriori we must have equality $T(v) = T(v^\sharp)$ as well, proving that $v^\sharp$  also minimizes \eqref{def:min_m}.

It remains to show that
\be \label{eq:v_good}
v(x) = e^{i \alpha} v^\sharp(x-x_0) \quad \mbox{for a.\,e.~$x \in \R^d$},
\ee
with some constants $\alpha \in \R$ and $x_0 \in \R^d$, which follows from Theorem \ref{thm:main} once we know hat $\widehat{v}$ is continuous. But the continuity of $\widehat{v}$ follows in an analogous fashion as in the proof of Theorem \ref{thm:GN} using also that $L$ has a continuous symbol $\omega(\xi)$ by assumption. Thus we conclude that \eqref{eq:v_good} holds. 

Finally, we deduce that $Q(x) = e^{i \alpha} Q^\sharp(x-x_0)$ for a.\,e.~$x \in \R^d$ from \eqref{eq:v_good} and using the simple fact that $(\mu v)^\sharp = |\mu| v^\sharp$ for any constant $\mu \in \C$.
\end{proof}

\subsection{Adams--Moser--Trudinger Type Inequalities in $\R^d$.}

As a further application of Theorem \ref{thm:main}, we will show radial symmetry for maximizers of {\bf Adams--Moser--Trudinger} type inequalities for the Sobolev spaces $H^{d/2}(\R^d)$ in any dimension $d \in \N$. More precisely, by following e.\,g.~\cite{Ru-05,LiRu-08,RuSa-13,LaLu-13}, we can obtain that
\be \label{def:S}
S_d(\alpha) := \sup_{{u \in H^{d/2}(\R^d) \atop \| u \|_{H^{d/2}} \leq 1}} \int_{\R^d} ( e^{\alpha|u|^2} - 1) \, dx = \begin{dcases*}\mbox{finite} & \mbox{for $0 < \alpha \leq \alpha_*$}, \\ +\infty & \mbox{for $\alpha > \alpha_*$} \end{dcases*}
\ee  
with some constant $\alpha_* = \alpha_*(d) > 0$. Here and in what follows, we choose the Sobolev norm so that
$$
\| u \|_{H^{d/2}}^2 = \| (-\DD)^{d/4} u \|_{L^2}^2 + \| u\|_{L^2}^2.
$$
The following discussion can also be extended to different choices for the norm $\| u \|_{H^{d/2}}$ and the relevant changes are left to the reader.

In $d=2$ dimensions, it shown in \cite{Is-11} that $\alpha_*(2)=4 \pi$ and that $S_{d=2}(\alpha_*)$ is attained (by making use of classical rearrangement techniques among other things; see also \cite{Ru-05}). For the existence of a critical constant $\alpha_*(d)> 0$ and its explicit value in any dimension $d$, we refer to \cite{LaLu-13}[Theorem 1.7] (although a different choice of $\| u \|_{H^s}$ is used and hence the numerical value of $\alpha_*(d)$ may have to be changed to our choice here.)

Now in view of Theorem \ref{thm:main}, we use the exponential series and monotone convergence theorem to find
\be \label{eq:exp}
\int_{\R^d} ( e^{\alpha |u|^2 } - 1) \, dx = \sum_{n=1}^\infty \frac{\alpha^n}{n!} \| u \|_{L^{2n}}^{2n}.
\ee
Thus we are in a situation where Theorem \ref{thm:main} may be used to deduce symmetry of maximizers for $S_d(\alpha)$. Indeed, we have the following general result.

\begin{theorem} \label{thm:ATM}
Let $d \in \N$ and $\alpha > 0$. Then every maximizer $Q \in H^{d/2}(\R^d)$ for $S_d(\alpha)$ has the form
$$
Q(x) = e^{i \alpha} Q^\sharp(x-x_0)
$$
with some constants $\alpha \in \R$ and $x_0 \in \R^d$. In particular, the function $Q(x)$ is radially symmetric and real-valued (up to translation and constant phase).
\end{theorem}

\begin{remarks*} {\em 
1) Since functions in $H^{d/2}(\R^d)$ are in general complex-valued, the assertion that maximizers $Q$ can be chosen real-valued (up to a constant phase) is non-trivial (in particular, for $d \geq 3$ when $u \in H^{d/2}(\R^d)$ does not guarantee that $|u| \in H^{d/2}(\R^d)$.)

2) As in the case of optimizers for \eqref{ineq:GN}, it is not hard to see that $\widehat{Q} \in L^1(\R^d)$ by using arguments in the spirit of the proof of Lemma \ref{lem:AMT} below. Thus $Q^\sharp : \R^d \to \R$ is a continuous positive definite function.

3) We hope that the Fourier rearrangement techniques presented here will also prove to be useful for showing the existence of minimizers for suitable $\alpha \in (0,\alpha_*]$ and any dimension $d \geq 1$.

4) For low dimensions $d \in  \{ 1,2 \}$, we can apply classical rearrangement techniques and maximum principle arguments to show that minimizers $Q$ must be non-negative (up to a constant phase). It would be interesting to prove (or disprove) that $Q$ can have sign changes in dimension $d \geq 3$.
}
\end{remarks*}

\begin{proof}[Proof of Theorem \ref{thm:ATM}]
Suppose $Q \in H^{d/2}(\R^d)$ is a maximizer for $S_d(\alpha)$. Clearly, we must have $\| Q \|_{H^{d/2}} = 1$. Since $Q \in H^{d/2}(\R^d)$ implies $\widehat{Q} \in L^{q}(\R^d)$ for any $1 < q \leq 2$, we can apply Theorem \ref{thm:main} and use \eqref{eq:exp} to conclude that
\be \label{eq:exp_S}
\int_{\R^d} ( e^{\alpha |Q|^2} - 1 ) \, dx \leq \int_{\R^d} ( e^{\alpha |Q^\sharp|^2 } - 1) \, dx.
\ee
On the other hand, we have that $\| Q^\sharp \|_{H^{d/2}} \leq \| Q \|_{H^{d/2}} = 1$ by Theorem \ref{thm:main}. Thus we conclude that $Q^\sharp$ is also a maximizer  and that we equality in \eqref{eq:exp_S} and $\| (-\DD)^{d/4} Q^\sharp \|_{L^2} = \| (-\DD)^{d/4} Q \|_{L^2}$.

To complete the proof by using Theorem \ref{thm:main}, we need to show that $\widehat{Q}$ is continuous. But this property follows from Lemma \ref{lem:AMT} below, where  the Euler-Lagrange equation for $Q$ is studied in Fourier space.
\end{proof}

\section{Some Possible Extensions and Remarks}

We mention here some directions in which the use of the Fourier rearrangement and the arguments in this paper can be extended, followed by some remarks on exponents $p \geq 2$ that are not even integers.

\subsection*{Weighted $L^p$-Norms} For $1 \leq p < \infty$ and $d \geq 1$, we define the weighted norms 
$$
\| f \|_{\dt{L}^{p}_\alpha} = \left ( \int_{\R^d} \frac{|f(x)|^p}{|x|^\alpha}  \, dx \right )^{1/p} \quad \mbox{with $0 < \alpha < d$},
$$
$$
\| f \|_{L^{p}_\alpha} =  \left ( \int_{\R^d} \frac{|f(x)|^p}{(1+|x|^2)^{\alpha/2}} \, dx \right )^{1/p} \quad \mbox{with $0 < \alpha < \infty$}.
$$
Indeed, a corresponding version of Theorem \ref{thm:main} remains true if $\| f \|_{L^p}$ is replaced by $\| f \|_{\dt{L}^{p}_\alpha}$ or $\| f \|_{L^{p}_\alpha}$, provided that $p \in 2 \N$ is an even integer. This is due to the fact that our arguments below will carry over in this case. Indeed, by using the Fourier transform we have, for $p \in 2 \N$,
$$
\| f \|_{\dt{L}^{p}_\alpha}^p =  \left  ( R_\alpha \ast \wf \ast \wcf  \ast \ldots \wf \ast \wcf \right )(0), \quad  \| f \|_{L^{p}_\alpha}^p =  \left ( G_\alpha \ast  \wf \ast \wcf  \ast \ldots \wf \ast \wcf \right )(0),
$$
with $p \in 2 \N$ convolutions appearing on the right-hand sides. Here $R_\alpha$ and $G_\alpha$ denote the kernels of the  corresponding Riesz and Bessel potentials, respectively. We recall the classical formulas
$$
R_\alpha(\xi) =  \frac{c_{\alpha,d}}{|\xi|^{d-\alpha}} \quad \mbox{and} \quad G_\alpha(\xi) = d_{\alpha, d} \int_{0}^\infty e^{-|\xi|^2/4 \pi \delta} e^{-\delta /4 \pi} \delta^{(-d + \alpha)/2} \frac{d \delta}{\delta} ,
$$
with some positive constants $c_{\alpha,d} > 0$ and $d_{\alpha,d} > 0$.  Since $R_\alpha=R_\alpha^*$ and $G_\alpha=G_\alpha^*$ are equal to their symmetric-decreasing rearrangements, we can directly apply the proof of Lemma \ref{lem:BLL} (using the Brascamp--Lieb--Luttinger inequality) to conclude that 
$$
\| f \|_{\dt{L}^{p}_\alpha} \leq \| f^\sharp \|_{\dt{L}^{p}_\alpha} \quad \mbox{and} \quad \| f \|_{L^{p}_\alpha} \leq \| f^\sharp \|_{L^{p}_\alpha}
$$
under some appropriate conditions on $f$ (e.\,g., $f$ belongs to $\mathcal{S}(\R^d)$). Furthermore, the classification of the equality case in Theorem \ref{thm:main} (assuming that $\wf$ is continuous) carries over to the weighted norms $\| f \|_{\dt{L}^{p}_\alpha}$ and $\| f \|_{L^{p}_\alpha}$ and equality occurs if and only if $f(x) = e^{i \alpha} f^\sharp(x)$ almost everywhere. (Note that no translation occurs since the weighted norms break translational invariance.) 

%As a potential application of weighted norms, we mention the family of {\em Caffarelli--Kohn--Nirenberg type inequalities} given by
%\be \label{ineq:CKN}
%\| |x|^{-a} f \|_{L^p} \leq C \| (-\Delta)^{s/2} f \|_{L^2}^\theta \| |x|^b f \|_{L^2}^{1-\theta}
%\ee
%with $s > 0$, $\theta=\theta(a,b,p,s,d) \in (0,1)$ and the parameters $(a,b,p)$ satisfy some suitable conditions depending on $s>0$ and $d \in \N$. We notice that for even integer $p \in 2 \N$ and $0 < a < d$ and $0 \leq b \leq 1$, the radial symmetry of optimizers for \eqref{ineq:CKN} can be studied by using Fourier rearrangement. Note also that, by Plancherel, we conclude from the known rearrangement inequality $\| (-\DD)^{s/2} u^* \|_{L^2} \leq \| (-\DD)^{s/2} u \|_{L^2}$ for $s \in [0,1]$ that it holds
%\be
%\| |x|^{b} f^\sharp \|_{L^2} \leq \| |x|^b f \|_{L^2}
%\ee
%for $f \in L^2(\R^d, |x|^{2b} \, dx)$ provided that $b \in [0,1]$. For deep recent results concerning symmetry breaking for Caffarelli--Kohn--Nirenberg type inequalities, we refer to \cite{

Finally, the above discussion can be readily generalized to weights other than $w(x) = |x|^{-a}$ or $w(x) = (1+|x|^2)^{-a/2}$. The only assumption to be made is that the weight function $w(x) \geq 0$ has a Fourier transform $\widehat{w} = \widehat{w}^*$ that is equal to its symmetric-decreasing rearrangement.

\subsection*{Generalized Choquard--Hartree Norms} For $1 \leq p < \infty$ and $0 <\alpha < d$, we consider the family of Choquard--Hartree type norms
$$
\| f \|_{D^{\alpha,p}} = \left ( \iint_{\R^d \times \R^d} \frac{|f(x)|^{p} |f(y)|^{p}}{|x-y|^\alpha} \, dx \, dy \right )^{1/2p}.
$$ 
For $(p,\alpha,d)=(2,1,3)$, such an expression occurs e.\,g.~in the Choquard--Pekar model describing polarons (see, e.\,g., \cite{Li-77}); see also  \cite{MoVs-17} for a current overview on generalized Choquard--Hartree type energy functionals. In the case of even integer $p \geq 2$, the norms $\| f \|_{D^{\alpha,p}}$ are amenable to the Fourier arguments used below. Indeed, for $p \in 2\N$, we get
$$
\| f \|_{D^{\alpha,p}}^p = c_{\alpha,d} \int_{\R^d} \frac{\overline{F(\xi)} F(\xi)}{|\xi|^{d-\alpha} } \, d\xi, \quad \mbox{where} \ \ F(\xi) = ( \wf \ast \wcf \ast \ldots \wf \ast \wcf)(\xi)
$$
with $p-1\in \N$ convolutions in the definition of $F$. Likewise, we can deduce that
$$
\| f \|_{D^{\alpha,p}} \leq \| f^\sharp \|_{D^{\alpha,p} }
$$
for  all $f \in \mathcal{S}(\R^d)$, say. Again, we can obtain a classification result for the case of equality \`a la Theorem \ref{thm:main}. Clearly, the techniques laid out in this paper also apply to show radial symmetry of minimizers of Choquard--Hartree type energy functionals
$$
\mathcal{E}(u) = \frac 1 2 \langle u, L u \rangle - \frac{1}{2p} \| u \|_{D^{\alpha,p}}^{2p},
$$
with some (pseudo-)differential operator $L$ satisfying Assumption \ref{ass:L}.

\subsection*{Sobolev/H\"ormander--Beurling Type Norms} For the Sobolev spaces $H^{s,p}(\R^d)$ with $p \neq 2$, the Fourier rearrangement $f \mapsto f^\sharp$ does not seem to produce useful rearrangement estimates right away. However, we can consider the Sobolev type spaces $\widehat{H}^{s,p}(\R^d)$ with $1 \leq p \leq \infty$ and $s \in \R$ that are defined as the completion of $\mathcal{S}(\R^d)$ with respect to the norm
$$
\| f \|_{\widehat{H}^{s,p}} =  \| (1+|\cdot|^2)^{\frac s 2}\widehat{f}   \|_{L^{p'}} \quad \mbox{with} \quad \frac{1}{p} + \frac{1}{p'} = 1.
$$
Clearly, we have the equality $\widehat{H}^{s,2}(\R^d) = H^{s,2}(\R^d)$ by Plancherel's theorem and the continuous embedding $\widehat{H}^{s,p}(\R^d) \subset H^{s,p}(\R^d)$ for any $p \geq 2$ by the Hausdorff--Young inequality (as well as the reversed embedding $H^{s,p}(\R^d) \subset \widehat{H}^{s,p}(\R^d)$ for $p \leq 2$.) We remark that the spaces $\widehat{H}^{s,p}(\R^d)$ coincide with the spaces $B_{p',k}(\R^d)$ which can be found in H\"ormander's book \cite{Ho-83}[Section 10.1] if we take the weight function $k=(1+|\xi|^2)^{\frac s 2}$. The spaces $B_{p',k}(\R^d)$ are also referred to as {\em H\"ormander--Beurling spaces}. For an application of the spaces $\widehat{H}^{s,p}(\R^d)$ in the study of Cauchy problems for dispersive nonlinear PDEs, see \cite{CaVeVi-01}.

By using the arguments in this paper, we obtain the general inequality
\be
\| f^\sharp \|_{\widehat{H}^{s,p}} \leq \| f \|_{\widehat{H}^{s,p}}
\ee
valid for any $f \in \widehat{H}^{s,p}(\R^d)$, where  $1 \leq p < \infty$ and $s \geq 0$. Moreover, for  strictly positive $s > 0$, equality holds if and only if $|\wf| = (\wf)^*$ almost everywhere. These assertions all follow from a straightforward application of the proof of Lemma \ref{lem:L_rearrange} below. Again, a strict rearrangement principle in the sense of Theorem \ref{thm:main} can be proven when $\langle u, L u \rangle$ is replaced with $\| f \|_{\widehat{H}^{s,p}}$.

%\subsection*{Fourier Rearrangement on $\T$}  For given function $f \in L^2(\T)$, we let $(a_n)$ denote  

%, we define its Fourier rearrangement as
%$$
%f^\sharp(t) = \sum_{n \in \mathbb{Z}} a_n^* e^{in t}, \quad t \in [0,2 \pi],
%$$
%where the $(a_n^*)_{n \in \mathbb{Z}}$ 

\subsection*{Non-Radial Fourier Multipliers} The ideas of the present paper can be extended to the case where rotational invariance of $L$ is weakened to cylindrical symmetry in $d \geq 2$ dimensions. More precisely, suppose the operator $L$ has a multiplier $\omega : \R^n \to \R$ that is {\em cylindrically symmetric} with respect to some direction $e \in \mathbb{S}^{d-1}$, i.\,e.
$$
\mbox{$\omega(R \xi) = \omega(\xi)$ for a.\,e.~$\xi \in \R^d$ and all rotations $R \in O(d)$ with $R e = e$.}
$$
By a global rotation of coordinates in $\R^d$, we may assume henceforth that $e=e_1 =(1,0,\ldots, 0) \in \mathbb{S}^{d-1}$.  Accordingly, we can define the {\bf partial Fourier rearrangement with respect to $e_1$} by setting
$$
f^{\sharp e_1} = \FF^{-1} \left \{ (\FF f)^{* e_1} \right \},
$$
where $g^{*e_1} : \R^d \to [0,\infty)$ denotes the function obtained by symmetric-decreasing rearrangement of $(x_2, \ldots, x_d) \mapsto g(x_1, x_2, \ldots, x_d)$ on $\R^{d-1}$ for each $x_1 \in \R$ fixed. It is easy to see that $f^{\sharp e_1}$ is cylindrically symmetric with respect to the $e_1$-axis.

In \cite{HiLeSo-18}, we will use this idea to prove cylindrical symmetry of {\em boosted ground states $Q_v \in H^s(\R^d)$} for general NLS of the form
$$
(-\Delta)^s Q + i v \cdot \nabla Q + \alpha Q - |Q|^{p-2} Q = 0 \quad \mbox{in $\R^d$},
$$
where $s \in[1/2, \infty)$, $d \geq 2$, $\alpha \in \R$, $2 < p < p_*(d,s)$, and $v \in \R^d$ (with $|v| < 1$ if $s=1/2$).

\subsection*{On Non-Even Integer Exponents $p$} It would be desirable to understand what can happen with $\| f \|_{L^p}$ under Fourier rearrangement for exponents $p \not \in  2 \N \cup \{ \infty \}$. For such exponents $p$, it can be shown that there exists $f \in \mathcal{S}(\R^d)$ satisfying the strict inequality
$$
\| \FF^{-1} ( |\FF f| ) \|_{L^p} < \| f \|_{L^p}.
$$
This fact is sometimes referred to as the failure of the upper Hardy-Littlewood majorant property for $L^p$-norms when $p \not \in 2 \N \cup \{\infty \}$. Indeed, the existence of such functions $f$ can be inferred from known counterexamples on the circle group $\T$ given by trigonometric polynomials $P \in L^2(\T)$; see \cite{Li-60, Bo-62, MoSc-09}. If we take $f(x) = \lambda^{1/2p} e^{-\pi \lambda x^2}P(x)$ and choose $0 < \lambda \ll 1$ sufficiently small, we can produce the desired examples $f \in \mathcal{S}(\R^d)$ for $d=1$. (Examples for $d \geq 2$ can be obtained in a similar way.) However, all the known examples do {\em not} satisfy $|\wf| = (\wf)^*$.  Hence it is still unclear whether counterexamples for the inequality $\| f \|_{L^p} \leq \| f^\sharp \|_{L^p}$ exist  at all for some $p \not \in 2 \N \cup \{ \infty \}$. 

Note that a relaxation of the above inequality is known to be true. In \cite{Montgomery76}, the author establishes the inequality $\| f\|_{L^p(\T)}\le 5\| f^{\sharp}\|_{L^p(\T)}$, $p>2$, where $f(x)\sim \sum_{n\ge 0} a_n\cos(2n\pi x)$ is an even function on $\big[-\tfrac{1}{2},\tfrac{1}{2}\big]_{\mathrm{per}}\simeq \T=\R/\mathbb{Z}$, and $f^{\sharp}(x)\sim \sum_{n\ge 0} a_n^*\cos(2n\pi x)$, where $a_0^*\ge a_1^*\ge \cdots$ is the decreasing reordering of the sequence $(|a_n|)_{n\ge 0}$. For an adaption of \cite{Montgomery76} from $\T$ to $\R^d$, see \cite{Frank2018}. Finally, we mention that the inequality $\| f \|_{L^p(\T)} \leq C_p \|  f^\sharp \|_{L^p(\T)}$ holds  with some constant $C_p \geq 1$ goes back to Hardy and Littlewood (see \cite[Chap.~XII,~Vol~II]{Zygmund_Trigo}). In \cite{Li-60}, Littlewood showed that we must have $C_p > 1$ when $2 <  p < \infty$ is not an even integer. But  it is unclear to us at the moment whether the result in \cite{Li-60} valid for $\T$ (with the discrete rearrangement of Fourier coefficients $a_n$) also carries over to $\R^d$ (with the symmetric-decreasing rearrangement of $\wf$).

\section{Preliminary Facts about the Fourier Rearrangement}

\label{sec:prelim}

In this section, we collect some basic properties of the Fourier rearrangement (but not all of them will be used in this paper). Furthermore, we will prove the non-strict versions of the inequalities stated in Theorem \ref{thm:main}.

\subsection{Basic Properties and Compactness} 
From the introduction above, we recall that for a function $f \in L^2(\R^d)$ we define its Fourier rearrangement as the function $f : \R^d \to \R$ to be given by
\be \label{def:f_sharp}
f^\sharp := \FF^{-1} \left \{ (\FF f)^* \right \} .
\ee
Here $\FF$ denotes the Fourier transform on $\R^d$ (with inverse $\FF^{-1}$), where we use the following convention 
\be
(\FF f)(\xi) \equiv \wf(\xi) = \int_{\R^d}  e^{-2 \pi i x \cdot \xi} f(x) \, dx
\ee
for any Schwartz function $f \in \mathcal{S}(\R^d)$. As usual, $\FF$ is extended to the space of tempered distributions $\mathcal{S}'(\R^d)$ by duality. For $q \in [1,\infty]$, we use $\FF (L^q(\R^d))$ to denote the space of all $f \in \mathcal{S}'(\R^d)$ such that $f=\wg$ for some $g \in L^q(\R^d)$. By the Hausdorff--Young inequality, we have $\FF(L^{p}(\R^d)) \subset L^{p'}(\R^d)$ for $p \in [1,2]$ with $1/p+1/p'=1$.

In \eqref{def:f_sharp} we use $g^* : \R^d \to [0, \infty)$ to denote the symmetric-decreasing rearrangement (or Schwarz rearrangement) of a measurable function $g : \R^d \to \C$ that vanishes at infinity (in the sense that the sets $\{ |g| > t \}$ have finite Lebesgue measure in $\R^d$ for every $t> 0$). Explicitly, we have the formula
\be
g^*(x) = \int_0^\infty \mathds{1}_{\{ |g| > t \}^*}(x) \, dt ,
\ee
where $A^* \subset \R^d$ stands for the symmetric-decreasing rearrangement of a measurable set $A \subset \R^d$ with finite measure $|A| < \infty$, i.\,e., $A^*=B_R(0)$ is the open $d$-dimensional ball in centered at the origin with radius $R \geq 0$ such that $|B_R(0)| = |A|$ (where we choose $A^* = \emptyset$ if $|A|=0$). We refer to \cite{LiLo-01} for a general discussion of the concept of symmetric-decreasing rearrangements of functions. 

From the classical fact that $\| g \|_{L^p} = \| g^* \|_{L^p}$ for $p \in [1, \infty]$ due to the equimeasurability of the symmetric decreasing-rearrangement, an application of Plancherel's theorem immediately yields
\be 
\| f^\sharp \|_{L^2} = \| f \|_{L^2}
\ee
for any $f \in L^2(\R^d)$. Also, it is easy to see that $f \mapsto f^\sharp$ is non-expansive on $L^2(\R^d)$, i.\,e.,
\be 
\| f^\sharp - g^\sharp \|_{L^2} \leq \| f-g \|_{L^2}
\ee
for all $f,g \in L^2(\R^d)$, which follows from Plancherel's theorem and the well-known non-expansivity estimate $\| u^* - v^* \|_{L^2} \leq \| u-v \|_{L^2}$. Furthermore, we find the following {\em Hardy--Littlewood type} inequality for the Fourier rearrangement:
\be \label{ineq:hardy}
\left | \int_{\R^d} \overline{f(x)} g(x) \,dx \right | \leq \int_{\R^d} f^\sharp(x) g^\sharp(x) \, dx
\ee
for every $f,g \in L^2(\R^d)$.  To prove \eqref{ineq:hardy}, we observe $|\langle f,  g \rangle | = | \langle \wf, \wg \rangle | \leq \langle |\wf|, |\wg| \rangle \leq \langle (\wf)^*, (\wg)^* \rangle = \langle f^\sharp, g^\sharp \rangle$ by using Plancherel's identity twice together with the classical Hardy--Littlewood inequality $\int_{\R^d} |u(x)| |v(x)| \, dx \leq \int_{\R^d} u^*(x) v^*(x) \, dx$. 

For later use, we also record the following compactness property of the set
\be 
H^{s,\sharp}(\R^d) := (H^s(\R^d))^\sharp = \{ f \in H^s(\R^d) : f=f^\sharp \}
\ee 
which consists of function in $H^s(\R^d)$ that are equal to their Fourier rearrangement. From Lemmas \ref{lem:wk_closed} and \ref{lem:compact} below, we have the following facts for all $d \in \N$:
\begin{itemize}
\item $H^{s,\sharp}(\R^d)$ is weakly closed in $H^{s}(\R^d)$ with $s \geq 0$.
\item $H^{s, \sharp}(\R^d)$ is compactly embedded in $L^p(\R^d)$ for every $2 < p < p_*(s,d)$ and $s>0$.
\end{itemize}
Evidently, these basic facts will greatly simplify an existence proof for optimizer in variational problems, where the Fourier rearrangement is beneficial.

\subsection{Non-Strict Rearrangement Inequalities}

We begin with the first inequality stated in Theorem \ref{thm:main}.

\begin{lemma} \label{lem:L_rearrange}
Suppose that $L$ satisfies Assumption \ref{ass:L} with some $s \geq 0$. Then, for all $f \in H^s(\R^d)$, we have
$$
\langle f^\sharp, L f^\sharp \rangle \leq \langle f, L f\rangle.
$$
Moreover, equality holds if and only if 
$$
|\widehat{f}(\xi)| = (\widehat{f})^*(\xi) \quad \mbox{for a.\,e.~$\xi \in \R^d$},
$$
where $(\widehat{f})^*:\R^d \to [0, \infty)$ denotes the symmetric-decreasing rearrangement of $\widehat{f} : \R^n \to \C$.
\end{lemma}

\begin{remark*}
{\em In particular, we obtain the general inequality
$$
\| (-\DD)^{s/2} f^\sharp \|_{L^2} \leq \| (-\DD)^{s/2} f\|_{L^2} \quad \mbox{for $f \in H^s(\R^d)$ and $s > 0$},
$$
and equality holds if and only if $|\wf|(\xi)= (\wf)^*(\xi)$ for almost every $\xi \in \R^d$.} 
\end{remark*}

\begin{proof}
We essentially follow the proof given in \cite{BoLe-17}, where the case $L=(-\Delta)^s$ with arbitrary $s >0$ is discussed. For the reader's convenience, we provide the straightforward adaptation to the more general case of operators $L$ satisfying Assumption \ref{ass:L}. 

With standard abuse of notation we write $\omega(|\xi|)$ using that $\omega : \R^d \to \R$ is radially symmetric on $\R^d$. Now, by Plancherel's theorem, the claimed inequality is equivalent to
\be 
\int_{\R^d} |(\widehat{f})^*(\xi)|^2 \omega(|\xi|) \, d \xi \leq \int_{\R^d} |\widehat{f}(\xi)|^2 \omega(|\xi|) \, d \xi.
\ee
From well-known properties of the symmetric-decreasing rearrangement we recall that $|(\widehat{f})^* |^2 = ( |\widehat{f}|^2)^*$. Hence it remains to show that
\be \label{ineq:g_rearr}
\int_{\R^d} g^*(\xi) \omega(|\xi|) \, d\xi  \leq \int_{\R^d} g(\xi) \omega(|\xi|) \, d \xi.
\ee
for any measurable nonnegative function $g :  \R^d \to [0, \infty )$ that vanishes at infinity. Moreover, by the layer-cake representation, we can write $g(\xi) = \int_{0}^\infty \chi_{\{ g > t \}}(\xi) \, dt$ for a.\,e.~$\xi$, and thus it suffices to prove that
\be  \label{ineq:measure}
\int_{\R^d} \chi_{A^*}(\xi) \omega(|\xi|) \, d \xi \leq \int_{\R^d} \chi_A(\xi) \omega(|\xi|) \, d \xi
\ee
for any measurable set $A\subset \R^n$ with finite measure, where $A^*$ denotes the symmetric-decreasing rearrangement of the set $A$, i.\,e., $A^*=B_R(0)$ is the open ball around 0 with radius $R \geq 0$ such that $\mu(B_R(0)) = \mu(A)$ (where we choose $R=0$ if $\mu(A)=0$.) It is now easy to see that \eqref{ineq:measure} holds. From $\mu(A \setminus A^*) = \mu(A)- \mu(A \cap A^*)$, $\mu(A^* \setminus A) = \mu(A^*) - \mu(A \cap A^*)$, and $\mu(A) = \mu(A^*)$, we deduce that $\mu(A \setminus A^*)=\mu(A^* \setminus A)$. From this and by using that $|\xi| \mapsto \omega(|\xi|)$ is monotone increasing, we get
\be \label{ineq:measure2}
\int_{A \setminus A^*} \omega(\xi) \, d \xi \geq \omega(R) \mu(A \setminus A^*) = \omega(R) \mu(A^* \setminus A) \geq \int_{A^* \setminus A} \omega(|\xi|) \, d \xi.
\ee 
Thus we find $\int_A \omega = \int_{A\setminus A^*} \omega + \int_{A \cap A^*} \omega \geq \int_{A^* \setminus A} \omega + \int_{A \cap A^*} \omega = \int_{A^*} \omega$, which proves that \eqref{ineq:measure} holds true. 

It remains to discuss the case when $f \in H^s(\R^d)$ satisfies $\langle f^\sharp, L f^\sharp \rangle = \langle f, L f \rangle$. Here we first note that in order to have equality in \eqref{ineq:measure} we must have that $\mu(A \setminus A^*)=0$. To see this, we argue by contradiction. So let us suppose we have equality in \eqref{ineq:measure} with $\mu(A \setminus A^*) > 0$. Equality in \eqref{ineq:measure} implies equality in \eqref{ineq:measure2}. Note that $\omega(|\xi|) < \omega(R)$ for $\xi \in A^*=B_R(0)$, since $|\xi| \mapsto \omega(|\xi|)$ is strictly increasing by assumption. Therefore, we have $\int_{A^* \setminus A} \omega < \omega(R) \mu(A \setminus A^*)$. But this contradicts that equality holds in \eqref{ineq:measure2}. This shows that equality in \eqref{ineq:measure} can hold only if $\mu(A \setminus A^*)=\mu(A^* \setminus A)=0$. Since $\mu(A) = \mu(A^*)$, the latter fact shows that $\mu(A \cap A^*) = \mu(A) = \mu(A^*)$ and hence the sets $A$ and $A^*$ coincide up to a set of zero measure.

By using the layer-cake representation for $g$, we deduce that equality in \eqref{ineq:g_rearr} can hold only if the characteristic functions satisfy $\chi_{\{ g > t \}}(\xi) = \chi_{\{ g>t \}^*}(\xi)$ for a.\,e.~$(\xi,t) \in \R^d \times \R_{>0}$. Hence, by the layer-cake principle, we conclude equality $g(\xi) = g^*(\xi)$ for a.\,e.~$\xi \in \R^d$. If we apply this to $g = |\widehat{f}|^2$, we deduce that the equality $\langle f^\sharp, L f^\sharp \rangle = \langle f, L f \rangle$ implies that $|\widehat{f}(\xi)|^2 = (|\widehat{f}(\xi)|^2)^* = |(\widehat{f})^*(\xi)|^2$ holds a.\,e. Since $(\widehat{f})^* \geq 0$ is nonnegative, this is equivalent to saying that $|\widehat{f}(\xi)| = (\widehat{f})^*(\xi)$ for a.\,e.~$\xi \in \R^d$. This completes the proof of Lemma \ref{lem:L_rearrange}.
\end{proof}

Now, we turn to the second inequality stated in Theorem \ref{thm:main}.
\begin{lemma} \label{lem:BLL}
Suppose $p \geq 2$ is an even integer or $p=\infty$. Let $p'$ denotes its dual exponent. Then, for every $f \in L^2(\R^d) \cap \FF(L^{p'}(\R^d))$, we have $f^\sharp \in L^p(\R^d)$ and 
$$
\| f \|_{L^p} \leq \| f^\sharp \|_{L^p}.
$$
\end{lemma}

\begin{proof}
If $p=2$, then we recall that $\|f \|_{L^2} = \| f^\sharp \|_{L^2}$ and there is nothing left to prove. So let us assume that $p>2$ holds for the rest of the proof.

Since $\widehat{f} \in L^{p'}(\R^d)$ by assumption and $\| (\widehat{f})^* \|_{L^{p'}} = \| \widehat{f} \|_{L^{p'}}$, the Hausdorff-Young inequality implies that $f^\sharp = \FF^{-1} \{ (\widehat{f})^* \}$ belongs to $L^p(\R^d)$.

Let us suppose that $p=2m$ with some integer $m \geq 2$. With the use of Lemma \ref{lem:conv} below we deduce that
$$
\| f \|_{L^p}^{p} = \FF(|f|^{2m})(0) = ( \wf \ast \wcf \ast \ldots \wf \ast \wcf)(0)
$$
with $2m-1$ convolutions on the right-hand side. From the Brascamp--Lieb--Luttinger \cite{BrLiLu-74} (which generalizes the classical Riesz inequality) we obtain
$$
 ( \wf \ast \wcf \ast \ldots \wf \ast \wcf)(0) \leq ( (\wf)^* \ast (\wf)^* \ast \ldots \ast (\wf)^* \ast (\wf)^*)(0) = \FF( |f^\sharp|^{2m})(0),
$$
where we also used that $(\wf)^* = (\wcf)^*$ due to the fact the function $\wcf(\xi) = \overline{f(-\xi)}$ has the same symmetric-decreasing rearrangement as $\wf(\xi)$. The completes the proof for $p \in 2 \N$.

Finally, let us assume that $p=\infty$ and thus $p'=1$. We readily see that
\be \label{eq:f_Lp_conv}
\| f \|_{L^\infty} \leq \int_{\R^d} |\widehat{f}(\xi)| \, d\xi \leq \int_{\R^d} (\widehat{f})^*(\xi) \, d\xi = f^\sharp(0).
\ee
On the other hand, since $\widehat{f}^*(\xi) \geq 0$ is non-negative and $(\widehat{f})^* \in L^1(\R^d)$, the function $f^\sharp$ is continuous and positive definite; in particular, it holds that $f^\sharp(0) \geq |f^\sharp(x)|$ for all $x \in \R^d$. Hence we find $f^\sharp(0) = \| f^\sharp \|_{L^\infty}$, which shows that $\|f \|_{L^\infty} \leq \| f^\sharp \|_{L^\infty}$.

\end{proof}

\section{On Equality in the Hardy--Littlewood Majorant Problem in $\R^d$ \\ and the Proof of Theorem \ref{thm:main}}

\label{sec:major}

We begin with the following result.

\begin{lemma} \label{lem:UMP}
Let $p \geq 2$ be an even integer or $p=\infty$ and let $1 \leq p' \leq 2$ be its dual exponent. Suppose that $f,g \in \FF(L^{p'}(\R^d))$ satisfy $|\widehat{f}(\xi)| \leq \widehat{g}(\xi)$ for a.\,e.~$\xi \in \R^d$. Then it holds 
$$
\| f \|_{L^p} \leq \| g \|_{L^p}.
$$
\end{lemma}

\begin{remark*}
{\em 
This result is a classical fact that can be traced back to the works of Hardy and Littlewood for the Fourier transform on the circle group $\mathbb{T}$ in \cite{HaLi-35}, and the corresponding property above is referred to as the {\bf upper majorant propety (UMP)} for the $L^p$-norms with $p \in 2 \N \cup \{\infty \}$. As we recall below, the assertion in Lemma \ref{lem:UMP} follows from the observation that for $p=2m$ with $m \in \N$, we can write $\| f \|_{L^p}$ using convolutions in Fourier space. (The proof for $p=\infty$ is also very straightforward.) However, the inequality $\| f \|_{L^p} \leq \| g \|_{L^p}$ when $|\widehat{f}(\xi)| \leq \widehat{g}(\xi)$ is known to {\em fail} in general for exponents $p \not \in 2 \N \cup \{ \infty \}$; we refer to \cite{Li-60, Bo-62, MoSc-09}  for explicit counterexamples on the torus $\mathbb{T}$. Using these examples, we recall that it is not hard to construct for any given $p \not \in 2 \N \cup \{ \infty \}$ and Schwartz function $f \in \mathcal{S}(\R^d)$ such that $\| \FF^{-1} (|\FF f|) \|_{L^p} < \| f \|_{L^p}$, showing that the upper majorant property fails for $L^p(\R^d)$-norms when $p \not \in 2 \N \cup \{ \infty\}$.
}
\end{remark*}

\begin{proof}
For the reader's convenience and also for later use below (when we study the case of equality), we provide a detailed proof of Lemma \ref{lem:UMP} as follows. We use a form of the proof, which will be more amenable to study the case of equality below.

First, we remark that $f, g \in L^p(\R^d)$ by the Hausdorff--Young inequality and the assumption that $\widehat{f}, \widehat{g} \in L^{p'}(\R^d)$. Let us first treat the case $p \neq \infty$. Thus we assume that $p = 2m$ with $m \in \N$. Correspondingly, the dual exponent is given by $p'=\frac{2m}{2m-1}$. Since $p=2m$, we observe that
\be \label{eq:conv_m}
\| f \|_{L^p}^p = \int_{\R^d} |f|^{2m} \, dx = \FF (|f|^{2m})(0) = ( \widehat{f} \ast \widehat{\overline{f}} \ast \cdots \ast \widehat{f} \ast \widehat{\overline{f}} )(0),
\ee
where the number of convolutions on the right side equals $m$. Recall that $\widehat{f} \in L^{\frac{2m}{2m-1}}(\R^d)$ by assumption. Thus the convolution is well-defined by Young's inequality and the map $\xi \mapsto ( \widehat{f} \ast \widehat{\overline{f}} \ast \cdots \ast \widehat{f} \ast \widehat{\overline{f}})(\xi)$ is bounded and continuous. Likewise, we find
 \be %\label{eq:conv_m}
\| g \|_{L^p}^p = ( \widehat{g} \ast \widehat{\overline{g}} \ast \cdots \ast \widehat{g} \ast \widehat{\overline{g}} )(0).
\ee
Since $\widehat{\overline{f}}(\xi) = \overline{\widehat{f}(-\xi)}$ and $\widehat{\overline{g}}(\xi) = \overline{\widehat{g}(-\xi)}= \widehat{g}(-\xi)$, the inequalities $|\widehat{f} | \leq \widehat{g}$ and $|\widehat{\overline{f}}| \leq \widehat{\overline{g}}$ hold a.\,e.~in $\R^d$. From the convolution expressions above we readily deduce that $\| f \|_{L^p}^p \leq \| g \|_{L^p}^p$ which is the desired bound for exponents $p \in 2 \mathbb{N}$.

Finally, we remark that the case $p=\infty$ can be treated in the same way as in the proof of Lemma \ref{lem:BLL}.
\end{proof}

\begin{lemma} \label{lem:strict}
Let $f, g$ be as in Lemma \ref{lem:UMP} above and suppose that $p > 2$ is an even integer or $p=\infty$. In addition, assume that $\widehat{f}$ is continuous and that $\{\xi : |\widehat{f}(\xi)| > 0 \} \subset \R^d$ is a connected set. Then equality 
$$
 \| f \|_{L^p} = \| g \|_{L^p}
 $$
holds if and only if
$$ 
\widehat{f}(\xi) = e^{i (\alpha + \beta \cdot \xi)} \widehat{g}(\xi) \quad \mbox{for all $\xi \in \R^d$},
$$
with some constants $\alpha \in \R$ and $\beta \in \R^d$.
\end{lemma}

\begin{remarks*}
{\em 1) If the set $\Omega=\{ |\widehat{f}| > 0 \}$ is {\em not} connected, the conclusion of Lemma \ref{lem:strict} may fail. For an explicit counterexample, we refer to the remark after the proof below.
We can still relax the condition though, as it would be enough to assume that the closure $\overline{\Omega}$ is connected, and that $\partial\Omega$ is locally a finite union of smooth graphs of co-dimension $1$ transverse to one another.  

2) In the proof of Theorem \ref{thm:main} below, we will apply this to the situation where $|\wf|= (\wf)^*$. In this case, the set $\{ |\wf| > 0 \}$ is always an open ball in $\R^d$ or all of $\R^d$ and therefore connected.

3) The conclusion of Lemma \ref{lem:strict} clearly fails for $p=2$, since we can take any measurable $\theta : \R^d \to \R$ and see that $\wf = e^{i \theta} \wg$ implies equality $\| f \|_{L^2} = \| g \|_{L^2}$ by Plancherel's theorem.} 
\end{remarks*}

\begin{proof}
The `if'-part is trivial to see. Indeed, we then have $f(x) = e^{i \alpha} g(x+x_0)$ with some $\alpha \in \R$ and $x_0=\frac{1}{2\pi} \beta \in \R^d$. Thus we get $\| f \|_{L^p} = \| g \|_{L^p}$.

To prove the `only if'-part, we divide the rest of the proof into the following steps.

\subsubsection*{Step 1} We first consider the case $p = 2m$ with some integer $m \geq 2$. Assume that $|\widehat{f}(\xi)| \leq \widehat{g}(\xi)$ holds a.\,e.~and suppose that we have equality $\| f \|_{L^p} = \| g \|_{L^p}$. Inspecting the proof of Lemma \ref{lem:UMP}, we readily see that equality $|\widehat{f}(\xi)| = \widehat{g}(\xi)$ must hold a.\,e.~in $\R^d$. Since $\widehat{f}$ is continuous by assumption, so is $\widehat{g}$, and thus equality $|\widehat{f}(\xi)| = \widehat{g}(\xi)$ holds everywhere.

Let us define the set
\be
\Omega = \{ \xi \in \R^d : |\widehat{f}(\xi)| > 0 \} ,
\ee
which is open by the continuity of $\widehat{f}$. In what follows, we assume that $\Omega \neq \emptyset$ is non-empty, since otherwise $\widehat{f} \equiv 0$ and the claim of Lemma \ref{lem:strict} clearly follows. Furthermore, we will make now the additional assumption that
$$
\mbox{{\em $\Omega \subset \R^d$ is simply connected.}}
$$
Later we will turn to the general case when $\Omega$ is only assumed to be connected. 

From standard arguments there exists a continuous function $\theta : \Omega \to \R$ such that
\be \label{eq:f_phase}
\widehat{f}(\xi) = e^{i \theta(\xi)} |\widehat{f}(\xi)| \quad \mbox{for $\xi \in \Omega$}.
\ee
Indeed, the mapping $z : \Omega \to \mathbb{S}^1$ with $z(\xi) = \widehat{f}(\xi)/|\widehat{f}(\xi)|$ is well-defined and continuous. Since $\Omega \subset \R^d$ is simply connected (by our assumption above) and the map $\R \to \mathbb{S}^1$, $t \mapsto  e^{it}$ is a universal covering, there exists a continuous function $\theta : \Omega \to \R$ such that $e^{i \theta(\xi)} = z(\xi)$ for all $\xi \in \Omega$. For the moment, let us extend the phase function $\theta$ to all of $\R^d$ by setting $\theta(\xi) = 0$ for $\xi \in \R^d \setminus \Omega$. In particular, we trivially have that $\widehat{f}(\xi) = e^{i \theta(\xi)} |\widehat{f}(\xi)|$ now is true for all $\xi \in \R^d$. Of course, the function $\theta$ fails to be continuous in general on all of $\R^d$.

Next, we notice that \eqref{eq:conv_m} can be written as
\be
\| f\|_{L^p}^p = ( A_f \ast A_f \ast \cdots \ast A_f)(0),
\ee
with $m-1$ convolutions on the right side and $A_f$ denotes the auto-correlation function of $\widehat{f}$ given by
\begin{align*}
A_f(\xi) & = (\widehat{f} \ast \widehat{\overline{f}})(\xi) =  (\widehat{f} \ast \overline{\widehat{f}(-\cdot)}  )(\xi) = \int_{\R^d} \widehat{f}(\xi+\xi') \overline{\widehat{f}(\xi')} \, d \xi' \\ 
& = \int_{\R^d} e^{i \{ \theta(\xi +\xi') - \theta(\xi')\}} |\widehat{f}(\xi+\xi')| |\widehat{f}(\xi')| \, d \xi'.
\end{align*}
We proceed to find with 
\begin{align*}
\| f \|_{L^p}^p & = ( A_f \ast A_f \ast \cdots \ast A_f)(0) \\
& = \int_{(\R^d)^{m-1}} A_f(0-\xi_{m-1}) A_f(\xi_{m-2} - \xi_{m-1}) \ldots A_f(\xi_2-\xi_1) A_f(\xi_1) \, d\xi_1 \ldots d \xi_{m-1} \\
& = m^{-d/2} \int_{\Sigma_{d,m-1}} A_f(\eta_1) A_f(\eta_2) \ldots A_f(\eta_m) \, \mathcal{H}^{d(m-1)}(d  \bm{\eta}) .
\end{align*}
Here $\Sigma_{d,m-1}$ denotes the $d (m-1)$-dimensional linear subspace of $(\R^d)^m= \R^{dm}$ that is given by
\be 
\Sigma_{d,m-1} = \left \{ \bm{\eta} = (\eta_1, \ldots, \eta_m) \in \R^{dm} : \sum_{k=1}^m \eta_k = 0 \right \} ,
\ee
and $\mathcal{H}^{d(m-1)}(d \bm{\eta})$ denotes the corresponding Hausdorff measure on $\Sigma_{d,m-1}$. To transform the integral over $(\R^d)^{m-1}$ to $\Sigma_{d,m-1}$ we made use of the linear diffeomorphism
$$
(\R^d)^{m-1}  \ni (\xi_1, \ldots, \xi_{m-1}) \longmapsto (\xi_1, \xi_2 - \xi_1, \ldots, \xi_{m-1} - \xi_{m-2}, -\xi_{m-1}) \in \Sigma_{d,m-1},
$$
whose Jacobian is easily to be seen the constant $m^{d/2}$. If we now recall \eqref{eq:f_phase}, we deduce
\be  \label{eq:f_good}
\| f \|_{L^p}^p =  m^{-d/2} \int_{\Sigma_{d,m-1} \times (\R^d)^m} F(\bm{\eta},\bm{\xi})  \, d \bm{\xi} \, \mathcal{H}^{d(m-1)}(d \bm{\eta})
\ee
where the function $F: \Sigma_{d,m-1} \times \R^{dm} \to \C$ is defined as
\be  \label{def:F}
F(\bm{\eta}, \bm{\xi}) =  e^{i \Theta(\bm{\eta}, \bm{\xi})} \prod_{k=1}^{m} |\widehat{f}(\eta_k + \xi_k)| |\widehat{f}(\xi_k)|, \quad \Theta(\bm{\eta}, \bm{\xi}) =  e^{i\sum_{k=1}^m \{ \theta(\eta_k + \xi_k) - \theta(\xi_k) \}},
\ee
with $\bm{\eta} = (\eta_1, \ldots, \eta_m) \in \Sigma_{d,m-1}$ and $\bm{\xi} = (\xi_1, \ldots, \xi_m) \in \R^{dm}$. Similarly, we obtain that 
\be 
\| g \|_{L^p}^p = m^{-d/2} \int_{\Sigma_{d,m-1} \times (\R^d)^m} |F(\bm{\eta},\bm{\xi})| \, d \bm{\xi} \, \mathcal{H}^{d(m-1)}(d \bm{\eta})
\ee
using that $\widehat{g} = |\widehat{f}|$. 

Now, since we have $\| f \|_{L^p} = \| g \|_{L^p}$, we deduce that equality holds in the triangle estimate $\int_{\Sigma_{d,m-1} \times \R^{dm}} F(\bm{\eta},\bm{\xi}) \leq \int_{\Sigma_{d,m-1} \times \R^{dm}} |F(\bm{\eta}, \bm{\xi})|$. Thus we conclude that 
\be \label{eq:F_triangle}
F(\bm{\eta},\bm{\xi})= |F(\bm{\eta}, \bm{\xi})| \quad \mbox{for a.\,e.~$(\bm{\eta}, \bm{\xi}) \in \Sigma_{d,m-1} \times \R^{dm}$}.
\ee
Let us now define the non-empty set $S \subset \Sigma_{d,m-1} \times \R^{dm}$, where $F$ does not vanish, i.\,e.,
\begin{align*}
S & = \left \{ (\bm{\eta}, \bm{\xi}) \in \Sigma_{d,m-1} \times \R^{dm} :  |F(\bm{\eta}, \bm{\xi}) | >  0 \right \} \\
& = \left \{ (\bm{\eta}, \bm{\xi}) \in \Sigma_{d,m-1} \times \R^{dm} :  \mbox{$(\eta_k + \xi_k, \xi_k) \in \Omega \times \Omega$ for $k=1, \ldots, m$} \right \} .
\end{align*}
We note in passing that the set $S$ may not be connected, although $\Omega \subset \R^d$ has this property by assumption.
From \eqref{eq:F_triangle} and \eqref{def:F} we now deduce that
\be  \label{eq:phase_master}
\Theta(\bm{\eta}, \bm{\xi}) = \sum_{k=1}^m \big [ \theta(\eta_k + \xi_k) - \theta(\xi_k) \big ] \in 2 \pi \mathbb{Z} 
\ee
for almost every $(\bm{\eta}, \bm{\xi}) \in S$.

\subsubsection*{Step 2} We claim that \eqref{eq:phase_master} implies that $\theta(\xi)$ is an affine function on the set $\Omega$, i.\,e.,
\be \label{eq:affine}
\theta(\xi) = \alpha + \beta \cdot \xi \quad \mbox{for $\xi \in \Omega$},
\ee
with some constants $\alpha \in \R$ and $\beta \in \R^d$. 

To prove \eqref{eq:affine}, we first show that identity \eqref{eq:phase_master} holds for all $(\bm{\eta}, \bm{\xi}) \in S$. Indeed, this can be deduced from a simple continuity argument as follows. We argue by contradiction. Suppose \eqref{eq:phase_master} fails to hold, i.\,e., there is some $(\bm{\eta}, \bm{\xi}) \in S$ such that $\Theta(\bm{\eta}, \bm{\xi}) = c$ for some real number $c \not \in 2 \pi \mathbb{Z}$. Since $(\bm{\eta}+\bm{\xi}, \bm{\xi}) \in \Omega^m \times \Omega^m$ and $\Omega \subset \R^d$ is open, there exists $r > 0$ such that $B_r(\bm{\eta} + \bm{\xi}) \times B_r(\bm{\xi}) \subset \Omega^m \times \Omega^m$. Thus the set 
$$
V_r = \{ \bm{\zeta} \in \Sigma_{d,m-1} : | \bm{\zeta} - \bm{\eta} | < r/2 \} \times B_{r/2}(\bm{\xi}) \subset S
$$ 
is an open neighborhood around $(\bm{\eta}, \bm{\xi}) \in S$ in $\Sigma_{d,m-1} \times \R^{dm}$. By the continuity of $\Theta : S \to \R$, we deduce that $\Theta(\bm{\eta}, \bm{\xi}) = c$ for every $(\bm{\eta}, \bm{\xi}) \in V_r$, provided that $r> 0$ is sufficiently small. Since the set $V_r \subset S$ has positive measure $\int_{V_r}  d \bm{\xi} \, \mathcal{H}^{d(m-1)}(d \bm{\eta}) > 0$, we get a contradiction to \eqref{eq:phase_master}. Thus  we conclude that \eqref{eq:phase_master} holds for all $(\bm{\eta}, \bm{\xi}) \in S$. Furthermore, by fixing $\bm{\xi}$ and taking the limit $\bm{\eta} \to 0$, the continuity of $\Theta$ implies that in \eqref{eq:phase_master} only the constant value zero can be attained. In summary, we have found that
\be  \label{eq:phase_master2}
\Theta(\bm{\eta}, \bm{\xi}) = \sum_{k=1}^m \big [ \theta(\eta_k + \xi_k) - \theta(\xi_k) \big ] = 0 \quad \mbox{for all $(\bm{\eta}, \bm{\xi}) \in S$.}
\ee

Next, let $q \in \Omega$ be given and take $h \in \R^d$ with $|h| < \mathrm{dist}(q, \pt \Omega)$ (with the convention that $\mathrm{dist}(q, \pt \Omega)=\infty$ if $\pt \Omega= \emptyset$). We consider the point $(\bm{\eta}, \bm{\xi}) = (\eta_1, \ldots, \eta_m, \xi_1, \ldots, \xi_m) \in S$ defined by
$$
\mbox{$\eta_1 = -\eta_2 = h$ and $\eta_k = 0$ for $k \geq 3$,} \quad \mbox{$\xi_1 = \xi_2 = q$ and $\xi_k = 0$ for $k \geq 3$.}
$$ 
Inserting this into \eqref{eq:phase_master2}, we obtain the identity
\be
\theta(q) = \frac{1}{2} \left ( \theta(q+h) + \theta(q-h) \right ) \quad \mbox{for  all $q \in \Omega$ and $h \in \R^d$ with $|h| < \mathrm{dist}(q, \pt \Omega)$,}
\ee
By Lemma \ref{lem:affine} below, we deduce that $\theta : \Omega \to \R^d$ must be an affine function, which is the desired claim \eqref{eq:affine}. 

Finally, we can redefine the phase function $\theta(\xi)$ on the complement $\Omega^c = \{ \xi \in \R^d : \widehat{f}(\xi) =  0 \}$ by setting $\theta(\xi) = \alpha + \beta \cdot \xi$ for $\xi \in \Omega^c$. Thus we obtain that
\be 
\widehat{f}(\xi) = e^{i ( \alpha + \beta \cdot \xi)} |\widehat{f}(\xi)| =  e^{i ( \alpha + \beta \cdot \xi)} \widehat{g}(\xi) \quad \mbox{for all $\xi \in \R^d$.}
\ee
This finishes the proof of Lemma \ref{lem:strict} for exponents $p \in 2 \mathbb{N}$ with $p > 2$, provided that $\Omega$ is simply connected.

\subsubsection*{Step 3} Let us now assume that the open set $\Omega \subset \R^d$ is connected (but not necessarily simply connected). We define the function $z : \R^d \to \mathbb{S}^1$ by setting $z(\xi) = \wf(\xi)/|\wf(\xi)|$ for $\xi \in \Omega$. Note that $z$ is continuous on $\Omega$.   

By adapting the arguments in Steps 1 and 2 above, we conclude that the phase of $z$ is {\em locally affine}. That is,  for every $\xi_0\in \Omega$ there exists some open ball $B_r(\xi_0) \subset \Omega$ such that 
$$
z(\xi) = e^{i (\alpha + \beta \cdot \xi)} \quad \mbox{for all $\xi \in B_r(\xi_0)$},
$$ 
with some constants $\alpha \in \R$ and $\beta \in \R^d$, which may depend on $B_r(\xi_0)$. But since $\Omega \subset \R^d$ is connected, it is elementary to see that $\beta=-iz^{-1}\nabla z$ and $e^{i\alpha}$  are global constants on $\Omega$. In other words:
$$
z(\xi) = e^{i ( \alpha + \beta \cdot \xi)} \quad \mbox{for all $\xi \in \Omega$}.
$$
Extending the function to all $\xi\in\R^d$, this completes the proof of Lemma \ref{lem:strict} for all even integers $p > 2$.

\subsubsection*{Step 4} We finally turn to the case $p=\infty$. Again, by inspecting the proof of Lemma \ref{lem:UMP} above, we see that $|\widehat{f}(\xi)| = \widehat{g}(\xi)$ must hold a.\,e.~in the case of equality $\| f\|_{L^\infty} = \| g \|_{L^\infty}$. As in Step 1 above, we let $\Omega = \{ \xi \in \R^d : |\widehat{f}(\xi)| > 0\}$. As before, we deduce that there is a continuous function $z : \Omega \to \mathbb{S}^1$ such that $\widehat{f}(\xi) = z(\xi) |\widehat{f}(\xi)|$ for all $\xi \in \Omega$. %For the moment, we  extend $z$ to all of $\R^d$ by setting $z(\xi) = 1$ for $\xi \in \Omega^c = \R^d \setminus \Omega$.

Since $\widehat{f} \in L^1(\R^d)$ by assumption, the function $f(x)$ is continuous and vanishes at infinity. In particular, there exists some $x_0 \in \R^d$ such that $\| f \|_{L^\infty} = \sup_{x \in \R^d} |f(x)| = |f(x_0)|$. Let us pick $\gamma \in \R$ such that $e^{i \gamma} f(x_0) = |f(x_0)|$. By the Fourier inversion formula,
\be
\| f \|_{L^\infty} = e^{i \gamma} f(x_0) = \int_{\R^d} e^{i \gamma} e^{2 \pi i x_0 \cdot \xi} \widehat{f}(\xi) \, d \xi =  \int_{\R^d} e^{i  \gamma} e^{2 \pi i x_0 \cdot \xi }  z(\xi)|\widehat{f}(\xi)| \, d \xi .
\ee
For the function $g$, we use that $\widehat{g} \in L^1(\R^d)$ is nonnegative to conclude that
\be 
\| g \|_{L^\infty} = g(0) = \int_{\R^d} \widehat{g}(\xi) \, d \xi.
\ee
Since $|\widehat{f}| = \widehat{g}$, we conclude that equality $\| f \|_{L^\infty}= \| g \|_{L^\infty}$ can occur only if 
\be
 e^{i  \gamma} e^{2 \pi i x_0 \cdot \xi }  z(\xi)|\widehat{f}(\xi)| = |\widehat{f}(\xi)| = \widehat{g}(\xi) \quad \mbox{for a.\,e.~$\xi \in \R^d$}.
\ee
In particular, this implies that $e^{i  \gamma} e^{2 \pi i x_0 \cdot \xi }  z(\xi) =1$ for a.e. $\xi \in \Omega$, or equivalently
\be
 	 z(\xi) =e^{-i\gamma-2i \pi x_0\cdot \xi} \quad \mbox{for a.\,e.~$\xi \in \Omega$. }%\gamma + 2 \pi x_0 \cdot \xi + \theta(\xi) \in 2 \pi \mathbb{Z} \quad \mbox{for a.\,e.~$\xi \in \Omega$.}
\ee
By the continuity of $z$ on $\Omega$, we deduce that the equality $z(\xi) =e^{-i\gamma-2i\pi x_0\cdot \xi}$ holds for all $\xi\in\Omega$.
Redefining $z$ on $\Omega^c$ accordingly, we get $z(\xi)=e^{i\theta(\xi)}$ with $\theta(\xi)=-2\pi x_0\cdot\xi-\gamma$. %, which is what wished to do.
%\be 
%\theta(\xi) = -\gamma - 2\pi x_0 \cdot \xi + 2\pi n \quad \mbox{for all $\xi \in \Omega$},
%\ee
%with some constant integer $n \in \mathbb{Z}$.  Thus, the function $\theta(\xi) = \alpha + \beta \cdot \xi$ for all $\xi \in \Omega$ with some constants $\alpha \in \R$ and $\beta \in \R^d$. By redefining $\theta$ on the complement $\Omega^c$ by setting $\theta(\xi) = \alpha + \beta \cdot \xi$ for $\xi \in \Omega^c$, we complete the proof of Lemma \ref{lem:strict}.

\end{proof}

\begin{remark*}
{\em {\bf (Counterexample when $\Omega = \{ |\wf| > 0 \}$ is not connected).}
We give an example to illustrate that the topological assumption in Lemma \ref{lem:strict} on $\Omega = \{ |\widehat{f}| > 0 \} \subset \R^d$ is not just for technical convenience. Here is a counterexample for $p=4$ when $\Omega$ is not connected, having two connected components separated by a sufficiently large distance.

Suppose $y \in \R^d$ is a given point with $|y| > 4$ and let $U = B_1(0) \cup B_1(y)$. We choose a function $\psi \in C^\infty_c(\R^d)$ with $\psi \geq 0$ such that $\mathrm{supp} \, \psi \subset U$ with $\psi |_{B_1(0)} \not \equiv 0$ and $\psi  |_{B_1(y)} \not \equiv 0$. Now we take numbers $\alpha, \beta \in \R$ with $\alpha \neq \beta$ and pick a smooth function $\vartheta : \R^d \to \R$ such that $\theta |_{B_1(0)} \equiv  \alpha$ and $\theta |_{B_1(y)} \equiv \beta$. Moreover, we define the function $\tilde{\psi} \in C^\infty_c(\R^d)$ as
$$
\tilde{\psi}(\xi) = e^{i \vartheta(\xi)} \psi(\xi).
$$

Next, we consider the functions $g = \FF^{-1} \psi$ and $f = \FF^{-1} \tilde{\psi}$. By construction, the function $\wf$ is continuous and we have $|\wf|= \wg = \psi \geq 0$ and the set $\Omega = \{ |\wf| > 0 \} = \{ \psi > 0 \}$ is not connected in $\R^d$. We now claim that
$$
\| f \|_{L^4} = \| g \|_{L^4}.
$$
Indeed, by  inspecting the proof of Lemma \ref{lem:strict} above and adapting the notation therein, this equality will follow if we can show that 
$$
\Theta(\bm{\eta}, \bm{\xi}) = e^{i \sum_{k=1}^2 \left \{ \vartheta(\eta_k +\xi_k) - \vartheta(\xi_k) \right \}} = 1 \quad \mbox{for all $(\bm{\eta}, \bm{\xi}) \in S$}.
$$
To see this, we note that $(\bm{\eta}, \bm{\xi}) =((\eta_1, \eta_2), (\xi_1, \xi_2)) \in S$ implies that 
$$
\eta_1 + \eta_2 = 0 \quad \mbox{and} \quad  (\eta_k + \xi_k, \eta_k)  \in \Omega \times \Omega \subset U \times U \ \ \mbox{for $k=1,2$}.
$$
Using that $U= B_1(0) \cup B_1(y)$ with $|y| > 4$, we can check that for all $(\bm{\eta}, \bm{\xi})  \in S$ we have
$$
\vartheta(\eta_1 + \xi_1) - \vartheta(\xi_1) + \vartheta(\eta_2 + \xi_2) - \vartheta(\xi_2) = 0,
$$
using that $\vartheta |_{B_1(0)} \equiv \alpha$ and $\vartheta |_{B_1(y)} \equiv \beta$. 
Hence we conclude that $\| f\|_{L^4} = \| g \|_{L^4}$ with $\wf(\xi) = e^{i \vartheta(\xi)} \wg(\xi)$, 
where the phase function $\vartheta : \R^d \to \R$ fails to be globally affine.
}
\end{remark*}

\subsection{Proof of Theorem \ref{thm:main}}

Suppose $p > 2$ is an even integer or $p=\infty$, and let $1 \leq p' < 2$ denote its conjugate exponent. Let $f \in H^s(\R^d) \cap \FF(L^{p'}(\R^d))$ be given. From Lemma \ref{lem:L_rearrange} and \ref{lem:BLL} we conclude
\be \label{ineq:f_proof}
\langle f^\sharp, L f^\sharp \rangle \leq \langle f, L f\rangle \quad \mbox{and} \quad \| f \|_{L^p} \leq \| f^\sharp \|_{L^p} .
\ee

It remains to prove the assertion about the equality case. Suppose that equality holds in both inequalities in \eqref{ineq:f_proof} and assume $f \not \equiv 0$ (for otherwise the claim is trivially true). By Lemma \ref{lem:L_rearrange}, this implies 
\be
|\widehat{f}(\xi)| = (\widehat{f})^*(\xi)  \quad \mbox{for a.\,e.~$\xi \in \R^d$},
\ee
where we recall that $(\widehat{f})^* : \R^d \to [0, \infty)$ denotes the symmetric-decreasing rearrangement of $\widehat{f}$. Since $\widehat{f}^*(\xi)$ is radially symmetric, monotone decreasing in $|\xi|$, and lower semi-continuous, the set $\Omega = \{ \xi \in \R^d : |\widehat{f}(\xi)| >0 \}$ is either an open ball $B_R(0) \subset \R^d$ around the origin, or else $\Omega = \R^d$ in case $\widehat{f}^*$ never vanishes. In either case the set $\Omega \subset \R^d$ is simply connected. Assuming now that $\widehat{f}(\xi)$ is continuous, we can apply Lemma \ref{lem:strict} with $\widehat{g}=(\widehat{f})^* \geq 0$ to deduce from $\| f \|_{L^p} = \| f^\sharp \|_{L^p} = \| g \|_{L^p}$ that we must have
\be 
\widehat{f}(\xi) = e^{i (\alpha + \beta \cdot \xi)} (\widehat{f})^*(\xi) \quad \mbox{for all $\xi \in \R^d$},
\ee
with some constants $\alpha \in \R$ and $\beta \in \R^d$. Thus we obtain that $f(x) = e^{i \alpha} f^\sharp(x-x_0)$ for a.\,e.~$x \in \R^d$ with the constant translation $x_0 = -\frac{1}{2 \pi} \beta \in \R^d$. 

The proof of Theorem \ref{thm:main} is now complete. \hfill $\qed$

\begin{appendix}

\section{Some Auxiliary Results}

\begin{lemma} \label{lem:affine}
Let $\Omega \subset \R^d$ be open and connected. Suppose $f : \Omega \to \R$ is a continuous function with the property 
$$
f(x) = \frac{1}{2} \left ( f(x+h) + f(x-h) \right ) \quad \mbox{for all $x \in \Omega$ and $h \in \R^d$ with $|h| < \mathrm{dist}(x, \pt \Omega)$}.
$$
Then $f : \Omega \to \R$ is an affine function, i.\,e., we have $f(x) = a + b \cdot x$ for all $x \in \Omega$ with some constants $a \in \R$ and $b \in \R^d$.
\end{lemma}

\begin{proof}
A possible proof of this lemma is quite elementary: For any $x_0 \in \Omega$, we show that in some open ball $B_r(x_0) \subset \Omega$, the function $f$ satisfies the affine functional relation for all dyadic weights 
$\tfrac{m}{2^n}\in [0,1]$. By continuity, this extends to all weight $p\in [0,1]$: thus $f$ is an affine function in $B_r(x_0)$. By the connectedness of $\Omega \subset \R^d$, we complete the proof.

For the reader's convenience, we shall now provide an alternative proof, which has more of a ``PDE flavor'', and which applies in fact to $f\in L^1_{\mathrm{loc}}(\Omega)$. 

Let $x_0 \in \Omega$ be given and suppose $B \subset  \Omega$ is an open ball with $x_0 \in B$ and closure $\overline{B} \subset \Omega$. Let $\phi \in C^\infty_c(B)$ be a test function. For any $h \in \R^d$ with $0 < |h| < \mathrm{dist}(B, \pt \Omega)$, we readily check that
\begin{multline*}
\int_{B} \left [ \frac{\phi(x+h) - 2 \phi(x) + \phi(x-h)}{|h|^2} \right ] f(x) \, dx \\ = \int_B \phi(x) \left [ \frac{f(x+h) - 2 f(x) + f(x-h)}{|h|^2} \right ] \, dx = 0,
\end{multline*}
where the last step follows from our assumption on $f$. On the other hand, for every unit vector $e \in \R^d$, we can choose $h = |h| e$ and find (by using dominated convergence and Taylor's theorem)
$$
\lim_{h \to 0} \int_B \left [ \frac{\phi(x+ |h| e) - 2 \phi(x) + \phi(x- |h| e)}{|h|^2} \right ] f(x) \, dx = \int_B ( e \cdot (D^2 \phi)(x) e) f(x) \, dx,
$$ 
where $(D^2 \phi)(x)$ denotes the Hessian matrix of $\phi$ evaluated at the point $x$. Since it holds
$$
\frac{\pt^2 \phi}{\pt x_i \pt x_j}(x) = \frac{1}{4} \left [ (e_i+e_j) \cdot (D^2 \phi)(x) (e_i+e_j))- (e_i - e_j) \cdot (D^2 \phi)(x) (e_i - e_j)\right ]
$$ 
thanks to the polarization formula for the symmetric matrix $D^2 \phi(x)$, we conclude
$$
\int_{B} \frac{\pt^2 \phi}{\pt x_i \pt x_j}(x) f(x) \,dx = 0 \quad \mbox{for $\phi \in C^\infty_c(B)$ and $1 \leq i,j \leq d$.}
$$ 
This shows that the identity $\frac{\pt^2 f}{\pt x_i \pt x_j} = 0$ holds for any $1 \leq i,j \leq d$ in the distributional sense on $B$, whence it follows that $f(x) = a + b \cdot x$ on $B$ with some constants $\alpha \in \R$ and $b \in \R^d$ (by a standard procedure using mollifiers). Since $x_0 \in \Omega$ was arbitrary and $\Omega \subset \R^d$ is connected, we conclude that $f(x) = a + b \cdot x$ on all of $\Omega$, with some constants $a\in \R$ and $b \in \R^d$. \end{proof}

%\newpage

\begin{lemma} \label{lem:AMT}
Let $d \in \N$ and $\alpha > 0$. If $u \in H^{d/2}(\R^d)$ is a maximizer for the Adams--Moser--Trudinger variational problem $S_d(\alpha)$ defined \eqref{def:S}, then $\widehat{u}$ is continuous.
\end{lemma}

\begin{proof}
By its maximizing property, the function $u \in H^{d/2}(\R^d)$ satisfies the corresponding Euler--Lagrange equation
\be
(-\Delta)^{d/2} u + u = \lambda u e^{\alpha |u|^2}
\ee
with some Lagrange multiplier $\lambda = \lambda(u) \in \R$. Integrating the equation against $\overline{u}$ and using that $\| u \|_{H^{d/2}}=1$, we obtain
\be 
\lambda = \frac{1}{\int_{\R^d} |u|^2 e^{\alpha |u|^2} \, dx} .
\ee
The rest of the proof will be divided into the two following steps.

\subsubsection*{Step 1: Upper Bound on $\lambda$} We show that the Lagrange multiplier satisfies 
\be \label{ineq:lam_apriori}
 \lambda < 1.
\ee
We argue as follows. Let $v = u^\sharp$ denote the Fourier rearrangement of the maximizer $u$. Since $u \in H^{d/2}(\R^d) \subset \FF(L^{q}(\R^d))$ for every $1 < q \leq 2$, we can apply Theorem \ref{thm:main} to conclude that
\be
\| v \|_{H^{d/2}} \leq \| u \|_{H^{d/2}}
\ee
and
\be \label{ineq:u_v}
 \int_{\R^d} (e^{\alpha |u|^2} -1) \, dx = \sum_{n=1} \frac{\alpha^n}{n!} \| u \|_{L^{2n}}^{2n} \leq \sum_{n=1} \frac{\alpha^n}{n!} \| v \|_{L^{2n}}^{2n} = \int_{\R^d} (e^{\alpha |v|^2} -1 ) \, dx .
\ee
Since $u$ is a maximizer, we must have equality everywhere above and in particular $v$ is also a maximizer for $S_d(\alpha)$. Therefore, we have
$$
(-\DD)^{d/2} v+ v  = \mu v e^{\alpha |v|^2}
$$
with some Lagrange multiplier $\mu= \mu(v) \in \R$ and likewise we obtain
$$
\mu = \frac{1}{\int_{\R^d} |v|^2 e^{\alpha |v|^2} \, dx }.
$$ 
But since we have equality in \eqref{ineq:u_v}, we have $\| u\|_{L^{2n}}^{2n} = \| v\|_{L^{2n}}^{2n}$ for any $n \in \N$. Hence, by series expansions, this implies $\int_{\R^d} |u|^2 e^{\alpha |u|^2} \, dx = \int_{\R^d} |v|^2 e^{\alpha |v|^2} \,dx$. Consequently, we obtain equality for the Lagrange multipliers:
\be 
\mu = \lambda.
\ee
Thus to prove the desired bound \eqref{ineq:lam_apriori}, we need to show that
\be  \label{ineq:mu}
\mu < 1.
\ee

To see that $\mu \geq 1$ cannot hold, we write the Euler-Lagrange equation satisfied $v$ as
\be  \label{eq:EL_trick}
(-\DD)^{d/2} v +( 1 -\mu ) v = F \quad \mbox{with $F=\mu v(e^{\alpha v^2} -1 )$},
\ee
where we use that $v^2 = |v|^2$ from now on, since $v$ is real-valued. Next, from \cite{OMeSe-08}[Lemma 2.2] we recall the following elementary estimate: For any $\beta > r > 1$, there exists a constant $C=C(\beta)>0$ such that
\be
\left ( e^{\alpha s^2} - 1 \right )^r \leq C(\beta) \left ( e^{\alpha \beta s^2} - 1 \right ) \quad \mbox{for all $s \in \R$.}
\ee
Since $\int_{\R^d} (e^{\gamma |w|^2} -1 ) \, dx < \infty$ for every $\gamma > 0$ and $w \in H^{d/2}(\R^d)$ (see Lemma \ref{lem:moser} below), we conclude from H\"older's inequality and the estimate above that $F = \mu v (e^{\alpha v^2}-1)$ belongs to $L^1(\R^d)$, which shows that $\widehat{F} \in C^0(\R^d)$ holds. Next, we write \eqref{eq:EL_trick} in Fourier space, which yields
\be \label{eq:EL_Fourier}
 \left (|2 \pi \xi|^{d}  +1 - \mu \right ) \widehat{v}(\xi)= \widehat{F}(\xi) .
\ee
Next, we will prove that
\be \label{ineq:F_positive}
\widehat{F}(\xi) > 0 \quad \mbox{for all $\xi \in \R^d$.}
\ee
First, we define the sequence $F_M \in L^1(\R^d)$, with $M \in \N$, by setting
$$
F_M(x) = \mu \sum_{k=1}^M \frac{\alpha^k}{k!} (v(x))^{2k+1},
$$ 
and we readily verify that $F_M \to F$ in $L^1(\R^d)$ as $M \to \infty$. As a consequence, we have $\widehat{F}_M(\xi) \to \widehat{F}(\xi)$ as $M \to \infty$ uniformly in $\xi \in \R^d$. Next, since $v \in H^{d/2}(\R^d)$, we recall that $\wv \in L^q(\R^d)$ for any $1 < q \leq 2$. Thus we can apply Lemma \ref{lem:conv} to conclude
$$
\widehat{F_M}(\xi) = \mu \sum_{k=1}^M \frac{\alpha^k}{k!} ( \underbrace{\wv \ast \ldots \ast \wv}_{\text{$2k+1$ times}})(\xi).
$$
Since $\wv(\xi) = (\wv)^*(\xi) \geq 0$ is non-negative together with $\mu > 0$ and $\alpha > 0$, we get
$$
0 \leq \widehat{F_1}(\xi) \leq \widehat{F_2}(\xi) \leq \ldots \leq \widehat{F_M}(\xi) \to \widehat{F}(\xi) \quad \mbox{for every $\xi \in \R^d$ as $M \to \infty$.}
$$
Note that the positivity of $\widehat{F}$ simply follows from \eqref{eq:EL_Fourier} and the form of $F$.  Since $\widehat{v}\neq 0$ is radial non-negative and monotone decreasing, there exists a nonempty ball $B_r(0)\subset \mathrm{supp}\,\widehat{v}$. For any $M\in\mathbb{N}$ we obtain $B_{Mr}(0)\subset \mathrm{supp}\, \widehat{F}_M$, and then $\mathrm{supp}\, \widehat{F}=\mathbb{R}^d$. In other words \eqref{ineq:F_positive} holds. If we recall now \eqref{eq:EL_Fourier}, we conclude that 
$$
\mu \le 1  \quad \mbox{and} \quad \mbox{$\wv(\xi) > 0$ for all $\xi \in \mathbb{R}^d$}. 
$$
Next, we assume that $\mu=1$ holds. From \eqref{eq:EL_Fourier} we get
$$
\wv(\xi) = \frac{\widehat{F}(\xi)}{|2 \pi \xi|^d} .
$$ 
Since $\widehat{F}(\xi) > 0$ is continuous, this implies
$$
|\wv(\xi)|^2 \geq  \frac{c}{|2 \pi\xi|^d} \quad \mbox{for} \quad |\xi| \leq r,
$$
with some small constants $c > 0$ and $r >0$. But this shows that $\wv \not \in L^2(\R^d)$, which is a contradiction. Thus we can only have that $\mu < 1$ holds and this proves $\eqref{ineq:mu}$.

\subsubsection*{Step 2: Conclusion} Recall that $u$ is an extremizer and thus satisfies in Fourier space
\be \label{eq:EL_finish}
\left (|2 \pi \xi|^d + 1 - \mu \right ) \widehat{u}(\xi) = \widehat{G}(\xi) \quad \mbox{almost everywhere},
\ee
with the function $G = \lambda u (e^{\alpha |u|^2}-1) \in L^1(\R^d)$ (analogous to Step 1 above). Thus the right-hand side is a continuous function. Furthermore, we recall that $\mu < 1$ from Step 1 above, whence it readily follows from \eqref{eq:EL_finish} that $\widehat{u} : \R^d \to \C$ must be a continuous function as well. The proof of Lemma \ref{lem:AMT} is now complete.
\end{proof}

\begin{lemma} \label{lem:moser}
For any $\gamma > 0$ and $u \in H^{d/2}(\R^d)$, we have $\int_{\R^d} (e^{\gamma |u|^2}-1) \, dx < \infty$.
\end{lemma}

\begin{proof}
We split the nonnegative integral as follows
$$
\int_{\R^d} (e^{\gamma |u|^2} - 1 ) \, dx = \int_{\{ |u| < 1 \} } (e^{\gamma |u|^2} - 1 ) \, dx + \int_{\{ |u| \geq 1 \}} (e^{\gamma |u|^2} - 1 ) \, dx =: \mathrm{I} + \mathrm{II}. 
$$
For $\mathrm{I}$, we note that
$$
\mathrm{I} = \sum_{k=1}^\infty \frac{\gamma^k}{k!} \left ( \int_{ \{ |u| < 1 \}} |u(x)|^{2k} \, dx \right ) \leq \sum_{k=1}^\infty \frac{\gamma^k}{k!} \| u \|_{L^2}^2 \leq e^{\gamma} \| u \|_{L^2}^2 < \infty.
$$
As for $\mathrm{II}$, we write $\Omega = \{ x \in \R^d: |u(x)| \geq 1 \}$ and we find
$$
\mathrm{II} \leq \int_{\Omega} e^{\gamma |u|^2} \,dx. % + |\Omega|.
$$ 
%Since $|\Omega | < \infty$, it thus remains to show that $\int_{\Omega} e^{\gamma |u|^2} \, dx < \infty$ holds. 
Let us show that the latter integral is finite. From \cite{LaLu-13}[Theorem 1.6] we deduce that there is some constant $\beta_0 > 0$ such that 
$$
\sup_{\| v \|_{H^{d/2}} \leq 1} \int_{\Omega} e^{\beta_0 |v|^2 } \,dx \leq C |\Omega|
$$ 
with some constant $C> 0$.  By scaling, it follows that for each $\beta > 0$ that
$$
\int_{\Omega} e^{\beta |v|^2} \, dx < \infty \quad \mbox{whenever} \quad \| v \|_{H^{d/2}} \leq \sqrt{\beta_0/\beta}.
$$
Now, we adapt the following trick in \cite{ChMa-85}. Let $\eps > 0$ be given and take $\varphi \in C^\infty_c(\R^d)$ such that $\| u - \varphi \|_{H^{d/2}} \leq \eps$. Using the pointwise  inequality $|u|^2 \leq 2 |u-\varphi|^2 + 2 |\varphi|^2 \leq 2 |u-\varphi|^2 + 2 \| \varphi \|_{L^\infty}^2$, we deduce that
$$
\int_{\Omega} e^{\gamma |u|^2} \, dx \leq e^{2 \gamma \| \varphi \|_{L^\infty}^2}  \int_{\Omega} e^{2 \gamma |u-\phi |^2 } \,dx \leq C e^{2 \gamma \| \varphi \|_{L^\infty}^2} |\Omega| < \infty,
$$ 
provided we take $\eps > 0$ and  $\varphi \in C^\infty_c(\R^d)$ such that $\| u - \varphi \|_{H^{d/2}} \leq \eps \leq \sqrt{\beta_0/2\gamma}$. 
\end{proof}

\begin{lemma} \label{lem:conv}
Let $m \in \N$ be an integer with $m \geq 1$. Then, for any $f \in \FF(L^{\frac{2m}{2m-1}}(\R^d))$,
$$
\widehat{(|f|^{2m})}(\xi) = (\wf \ast \wcf \ast \ldots \ast \wf \ast \wcf)(\xi) \quad \mbox{for all $\xi \in \R^d$},
$$
with $2m-1$ convolutions on the right side. 

Also, if $m \geq 2$ and $g \in \FF(L^{\frac{m}{m-1}}(\R^d))$ is real-valued, then
$$
\widehat{(g^{m})}(\xi) = (\wg  \ast  \ldots  \ast \wg)(\xi) \quad \mbox{for all $\xi \in \R^d$},
$$
with $m-1$ convolutions on the right side.
\end{lemma}

\begin{remark*} {\em 
Note that if $f \in \FF(L^{\frac{2m}{2m-1}}(\R^d))$ then $f \in  L^{2m}(\R^d)$ by the Hausdorff--Young inequality, whence the Fourier transform of $|f|^{2m} \in L^1(\R^d)$ is continuous and bounded. }
\end{remark*}

\begin{proof}
By the classical convolution theorem for Schwartz functions, the asserted identity clearly holds for any Schwartz function $f \in \mathcal{S}(\R^d)$. Now let $f \in \FF(L^{\frac{2m}{2m-1}}(\R^d))$ be given. By Young's inequality and $\wf \in L^{\frac{2m}{2m-1}}(\R^d)$, we readily check that the map $\xi \mapsto (\wf \ast \wcf \ast \ldots \ast \wf \ast \wcf)(\xi)$ is a bounded function and, moreover, it is elementary to see that this map is continuous and vanishes as $|\xi| \to \infty$.

Now we take a sequence $f_j \in \mathcal{S}(\R^d)$ such that $\wf_j \to \wf$ in $L^{\frac{2m}{2m-1}}(\R^d)$ as $j \to \infty$. Applying Young's inequality for convolutions, we deduce that
$$
(\wf_j \ast \wcf_j \ast \ldots \ast \wf_j \ast \wcf_j)(\xi) \to (\wf \ast \wcf \ast \ldots \ast \wf \ast \wcf)(\xi)
$$
for every $\xi \in \R^d$. On the other hand, by the Hausdorff-Young inequality, we have $f_j \to f$ in $L^{2m}(\R^d)$, whence it follows that $|f_j|^{2m} \to |f|^{2m}$ in $L^1(\R^d)$. Therefore, we conclude that $\widehat{(|f_j|^{2m})}(\xi) \to \widehat{(|f|^{2m})}(\xi)$ for every $\xi \in \R^d$. By this limiting argument, we deduce that the asserted identity also holds for general $f \in \FF(L^{\frac{2m}{2m-1}}(\R^d))$.

The proof for the claimed identity for $g \in \FF(L^{\frac{m}{m-1}}(\R^d))$ follows along the same lines and hence we omit the details.
\end{proof}

\begin{lemma} \label{lem:conti}
Let $d \in \N$ and $s \geq 0$. Then the Fourier rearrangement $f \mapsto f^\sharp$ is continuous on $H^s(\R^d)$, i.\,e., if $f_k \to f$ strongly in $H^s(\R^d)$ then $f_k^\sharp \to f^\sharp$ strongly  in $H^s(\R^d)$. 
\end{lemma}

\begin{proof}
Suppose that $f_k \to f$ in $H^s(\R^d)$. By Plancherel, this means that $\widehat{f}_k \to \widehat{f}$ in $L^2(\R^d, (1+|2\pi \xi|^2)^s d \xi)$. By standard arguments, we can find a subsequence $(\widehat{f}_{k_m})$ such that $\widehat{f}_{k_m} \to \widehat{f}$ pointwise a.\,e.~and $|\widehat{f}_{k_m}(\xi)| \leq \widehat{F}(\xi)$ for almost every $\xi \in \R^d$ with some nonnegative function $\widehat{F} \in L^2(\R^d, (1 + |2 \pi|^2)^s d\xi)$. Using  the bound
$$
\chi_{\{|\wf_{k_m}| > t\}}^*(x) \leq \chi_{\{ \widehat{F} > t \}}^*(x) \leq  \begin{dcases*} 1 & for $|x| \lesssim \| \widehat{F} \|_{L^2} / t^2$ \\ 0 & else \end{dcases*}
$$
and $|\wf_{k_m}| \to |\wf|$ pointwise a.\,e., we can use the dominated convergence theorem to deduce
$$
(\widehat{f}_{k_m})^*(\xi) = \int_0^\infty \chi_{\{|\widehat{f}_{k_m}| > t\}}^*(\xi) \, dt \rightarrow (\widehat{f})^*(\xi) = \int_0^\infty \chi_{\{ |\widehat{f}| > t \}}^*(\xi) \,dt
$$
pointwise for almost every $\xi \in \R^d$. Now, by Plancherel and dominated convergence again, 
$$
\| (f_{k_m})^\sharp - f^\sharp \|_{H^s}^2 = \int_{\R^d} |(\widehat{f}_{k_m})^*(\xi) - (\widehat{f})^*(\xi)|^2 \, (1+|2 \pi \xi|^2)^s \, d\xi \rightarrow 0.
$$
Thus we have shown that for every sequence $f_k \to f$ strongly in $H^s(\R^d)$ there exists a subsequence $(f_{k_m})$ such that $f_{k_m}^\sharp \to f^\sharp$ strongly in $H^s(\R^d)$. This proves that the mapping $f \mapsto f^\sharp$ is continuous on $H^s(\R^d)$.
\end{proof}

\section{Compactness Properties of the Fourier Rearrangement}

\label{sec:compact}

We recall that $H^{s,\sharp}(\R^d) = \{ f \in H^s(\R^d) : f=f^\sharp \}$ denotes the set of functions in $H^s(\R^d)$ that are equal to their Fourier rearrangement. 

\begin{lemma} \label{lem:wk_closed}
Let $d \in \N$ and $s \geq 0$. Then the set $H^{s,\sharp}(\R^d)$ is a closed convex cone in $H^s(\R^d)$. In particular, the set $H^{s, \sharp}(\R^d)$ is a weakly closed subset of $H^s(\R^d)$.
\end{lemma}

\begin{proof}
This follows from elementary arguments. Let us denote $K = H^{s,\sharp}(\R^d)$. First, we readily check that $K$ is a cone, i.\,e., we have $t f \in K$ for any $f \in K$ and $t \in [0,\infty)$. Indeed, this just follows from the identity $(\alpha f)^\sharp = \FF^{-1} ( |\FF (\alpha f)|^* )= \FF^{-1} (|\alpha| |\FF f|^*) = |\alpha| \FF^{-1}(|\FF f|^*) = |\alpha| f^\sharp$ valid for any constant $\alpha \in \C$. 

Next, we verify that the cone $K$ is convex. Suppose that $f, g \in K$ and let $\lambda \in [0,1]$. We need to show that $h := \lambda f + (1-\lambda) g \in K$. To see this, note that the functions $\wf = (\wf)^* \geq 0$ and $\wg = (\wg)^* \geq 0$ are radial monotone decreasing and lower-semicontinuous. Clearly, the convex combination $\widehat{h}= \lambda \wf + (1- \lambda) \wg$ has the same properties. But this implies that $\widehat{h}$ is equal to its symmetric-decreasing rearrangement $\widehat{h}^*$, whence it follows that $h = h^\sharp \in K$.

Finally, let $f_k \in K$ be a sequence with $f_k \to f$ strongly in $H^s(\R^d)$. By Lemma \ref{lem:conti}, we conclude that $f=f^\sharp$. This shows that $K$ is a closed subset of $H^s(\R^d)$. Since $K$ is closed and convex in $H^s(\R^d)$, it must be weakly closed.
\end{proof}

We have the following compactness lemma.

\begin{lemma} \label{lem:compact}
Let $d \in \N$ and $s > 0$. Then $H^{s, \sharp}(\R^d)$ is compactly embedded in $L^p(\R^d)$ for every $2 < p < p_*$, where $p_*=2d/(d-2s)$ if $s < d/2$ and $p_*= \infty$ if $s \geq d/2$. More precisely, if $(f_k)$ is a bounded sequence in $H^{s}(\R^d)$ such that $f_k = f_k^\sharp$ for all $k \in \N$, then there exist a subsequence $(f_{k_m})$ and some $f \in H^{s,\sharp}(\R^d)$ such that 
$$
\mbox{$f_{k_m} \to f$ strongly in $L^p(\R^d)$ as $m \to \infty$}
$$
for every $2 < p < p_*(s,d)$.
\end{lemma}

\begin{remarks*}
{\em 1) In fact, we will prove slightly more: We conclude strong convergence of $\widehat{f}_k=(\widehat{f}_k)^*$ in $L^{p'}(\R^d)$, where $1 < p' <2$ is the conjugate exponent of $2<p<p_*(s,d)$.

2) For $s>1/2$ and dimensions $d\geq 2$, we could alternatively show the compactness of the embedding $H^{s,\sharp}(\R^d) \subset L^p(\R^d)$ for $2 < p < p_*(s,d)$ by using a generalized Strauss lemma from \cite{ChOz-09} for radial functions in $\dot{H}^s(\R^d)$ with $1/2 < s < d/2$. However, by exploiting the additional property that the functions $\wf_k$ are radially monotone decreasing, we can also deal with the $d=1$ case and all $s >0$.}
\end{remarks*}

\begin{proof}
Let $(f_k)$ be a bounded sequence in $H^s(\R^d)$ with $f_k = f_k^\sharp$ for all $k$. By passing to a subsequence if necessary, we can assume that $f_k \weakto f$ weakly in $H^s(\R^d)$. By Lemma \ref{lem:wk_closed}, the weak limit $f$ belongs to $H^{s, \sharp}(\R^d)$. 

The rest of the proof is inspired by a compactness argument \cite{Li-77} for minimizing symmetric-decreasing sequences $u_k = u_k^* \in H^1(\R^3)$ for the Choquard--Pekar problem, the main ingredient being Helly's selection principle. Here we use a similar argument applied on the Fourier side. 

Using that $\| f_k \|_{H^s} \leq C$ with some constant $C>0$ combined with the fact that the functions $\wf_k(\xi) = (\wf_k)^*(\xi) \geq 0$ are radial and monotone decreasing, we deduce 
$$
C   \geq \int_{|\xi| \leq R} |\wf_k(\xi)|^2 (1 +|\xi|^2)^s \, d \xi \geq |\wf_k(R)|^2 \int_{|\xi| \leq R} (1+|\xi|^2)^s d\xi  \gtrsim |\wf_k(R)|^2 (  R^d + R^{d+2s} )
$$
for all radii $R > 0$. From this we conclude the uniform pointwise bound
\be \label{ineq:f_upper}
0 \leq\wf_k(\xi) \leq C \min \left \{ R^{-d/2}, R^{-d/2-s} \right \} \quad \mbox{for} \quad |\xi|=R.
\ee
Next, by invoking Helly's selection principle for the radially symmetric and monotone decreasing functions $(\wf_k)$, we can find a subsequence (still indexed by $k$ for convenience) such that we have pointwise convergence
$$
\wf_k(\xi) \to \widehat{g}(\xi) \quad \mbox{for every $\xi \in \R^d \setminus \{0 \}$ as $k \to \infty$},
$$ 
with some radial monotone decreasing function $0 \leq \widehat{g}(\xi) \leq C \min \{ |\xi|^{-d/2}, |\xi|^{-d/2-s} \}$. Thus, by using dominated convergence and estimate \eqref{ineq:f_upper}, we conclude that
$$
\| \wf_k - \widehat{g} \|_{L^r} \to 0 \quad \mbox{for every $q_* < r < 2$ as $k \to \infty$},
$$
where $q_* = 2d/(d+2s)$ if $s<d/s$ and $q_*=1$ if $s \geq d/2$ denotes the dual exponent of $p_*$. Note that the dominating function $C \min \{ |\xi|^{-d/2}, |\xi|^{-d/2-s} \}$ is in $L^r(\R^d)$.
An application of the Hausdorff--Young inequality yields
$$
\| f_k - g \|_{L^p} \to 0 \quad \mbox{for every $2 < p < p_*$ as $k \to \infty$.} 
$$
Finally, we note that $f_k \weakto f$ weakly in $H^s(\R^d)$ readily implies that $g=f$.
\end{proof}

Let $d \in \N$, $s > 0$, and suppose $2 < p < p_*(s,d)$. For the Gagliardo--Nirenberg inequality \eqref{ineq:GN} in Section \ref{sec:applications} above, we consider the corresponding Weinstein functional given by
\be
J_{d,s,p}(u) = \frac{ \| (-\DD)^s u \|_{L^2}^\theta \| u \|_{L^2}^{1- \theta}}{\| u \|_{L^p}}
\ee
for $u \in H^s(\R^d) \setminus \{ 0 \}$.

\begin{proposition}
For $d,s,p$ as above, the infimum
$$
\Copt_{d,s,p}^{-1} = \inf_{u \in H^s, u \not \equiv 0} J_{d,s,p}(u)
$$
is attained.
\end{proposition}

\begin{proof}
This can be deduce from e.\,g.~concentration-compactness methods; see \cite{BeFrVi-14} for an alternative proof based on an extension to $H^s(\R^d)$ of a compactness lemma due to Lieb.

Here we will give an alternative and quite elementary proof by using Fourier rearrangement provided that $p \in 2 \N$ is an even integer.

Let $(f_k) \subset H^s(\R^d)$ be a minimizing sequence and we can choose that $f_k = f_k^\sharp$ by Theorem \ref{thm:main}. Furthermore, by scaling, we can assume that $\| f_k \|_{L^2} = \| f_k \|_{L^p} = 1$ for all $k \in \N$. Since $J_{d,s,p}(f_k) \to \Copt_{d,s,p}^{-1}$, is is easy to see that $\| (-\DD)^{s/2} f_k \|_{L^2} \leq C$ for some constant $C> 0$. Thus $(f_k)$ is bounded in $H^s(\R^d)$ and we can assume (by passing to a subsequence if necessary) that $f_k \weakto f$ weakly in $H^s(\R^d)$. By Lemma \ref{lem:compact}, we conclude after passing to a subsequence that $f_k \to f$ strongly in $L^p(\R^d)$. Since $\| f_k \|_{L^p}=1$, this implies that $f \not \equiv 0$. Finally, using that $\liminf_{k} \| (-\DD)^{s/2} f_k \|_{L^2} \geq \| (-\DD)^{s/2} f \|_{L^2}$ and $\liminf_k \| f_k \|_{L^2} \geq \| f \|_{L^2}$ and the strong convergence $f_k \to f \not \equiv 0$ in $L^p(\R^d)$, we deduce 
$$
\Copt_{d,s,p}^{-1} = \lim_{k \to \infty} J_{d,s,p}(f_k) \geq J_{d,s,p}(f) \geq \Copt_{d,s,p}^{-1}.
$$
Hence we have equality everywhere and we see that $f \in H^s(\R^d) \setminus \{ 0 \}$ is a minimizer.
\end{proof}

\begin{proposition}
The infimum in \eqref{def:min_m} is attained.
\end{proposition}

\begin{proof}
We use the notation in the proof of Theorem \ref{thm:gs}. First, we easily check that
$$
\mathfrak{m} = \inf \{ T(v) : v \in H^s(\R^d), \; \| v \|_{L^p} = 1 \} = \inf \{ T(v) : v \in H^s(\R^d), \; \| v \|_{L^p} \geq 1 \}.
$$
By Theorem \ref{thm:main}, we have $T(v^\sharp) \leq T(v)$ and $\|v^\sharp \|_{L^p} \geq \| v \|_{L^p}$. Thus if $(v_k)$ is a minimizing sequence for $\mathfrak{m}$, so is the sequence $(v_k^\sharp)$. 

Thus we can assume that $v_k=v^\sharp_k$ is a minimizing sequence. Since $T(v_k) \geq C \| v_k \|_{H^s}$ with some constant $C> 0$ independent of $v_k$, we conclude that $(v_k)$ is bounded in $H^s(\R^d)$. By replacing $v_k \to v_k / \|v_k\|_{L^p}$ if necessary, we can assume that $\| v_k \|_{L^p}=1$ without loss of generality. Next, by standard arguments, we can assume that $v_k \weakto v$ weakly in $H^s(\R^d)$ for some $v \in H^s(\R^d)$. Furthermore, by Lemma \ref{lem:compact}, we conclude that $v_k \to v$ strongly in $L^p(\R^d)$ and hence $\| v \|_{L^p}=1$. By the weak lower semi-continuity of $T(v)$, we conclude that  $\mathfrak{m} = \lim T(v_k) \geq	T(v)  \geq \mathfrak{m}$
and thus $v$ is a minimizer.
\end{proof}

\end{appendix}

\bibliographystyle{plain}
%\bibliography{FourierRearrangementBiblio}

\begin{thebibliography}{10}

\bibitem{AlLi-89}
Frederick~J. Almgren, Jr. and Elliott~H. Lieb.
\newblock Symmetric decreasing rearrangement is sometimes continuous.
\newblock {\em J. Amer. Math. Soc.}, 2(4):683--773, 1989.

\bibitem{BeFrVi-14}
Jacopo Bellazzini, Rupert~L. Frank, and Nicola Visciglia.
\newblock Maximizers for {G}agliardo-{N}irenberg inequalities and related
  non-local problems.
\newblock {\em Math. Ann.}, 360(3-4):653--673, 2014.

\bibitem{Bo-62}
R.~P. Boas, Jr.
\newblock Majorant problems for trigonometric series.
\newblock {\em J. Analyse Math.}, 10:253--271, 1962/1963.

\bibitem{BoCaGoJe-18}
D.~{Bonheure}, J.-B. {Casteras}, T.~{Gou}, and L.~{Jeanjean}.
\newblock {Normalized solutions to the mixed dispersion nonlinear
  {Schr\"odinger} equation in the mass critical and supercritical regime}.
\newblock {\em ArXiv e-prints}, February 2018.
\newblock arXiv:1802.09217.

\bibitem{BoNa-15}
Denis Bonheure and Robson Nascimento.
\newblock Waveguide solutions for a nonlinear {S}chr\"odinger equation with
  mixed dispersion.
\newblock In {\em Contributions to nonlinear elliptic equations and systems},
  volume~86 of {\em Progr. Nonlinear Differential Equations Appl.}, pages
  31--53. Birkh\"auser/Springer, Cham, 2015.

\bibitem{BoLe-17}
Thomas Boulenger and Enno Lenzmann.
\newblock Blowup for biharmonic {NLS}.
\newblock {\em Ann. Sci. \'Ec. Norm. Sup\'er. (4)}, 50(3):503--544, 2017.

\bibitem{BrLiLu-74}
H.~J. Brascamp, Elliott~H. Lieb, and J.~M. Luttinger.
\newblock A general rearrangement inequality for multiple integrals.
\newblock {\em J. Functional Analysis}, 17:227--237, 1974.

\bibitem{BrZi-88}
John~E. Brothers and William~P. Ziemer.
\newblock Minimal rearrangements of {S}obolev functions.
\newblock {\em J. Reine Angew. Math.}, 384:153--179, 1988.

\bibitem{HiLeSo-18}
Lars Bugiera, Dominik Himmelsbach, Enno Lenzmann, and J{\'e}r{\'e}my Sok.
\newblock Symmetry of boosted ground states for {NLS} with general dispersion.
\newblock Work in preparation.

\bibitem{CaVeVi-01}
Thierry Cazenave, Luis Vega, and Mari~Cruz Vilela.
\newblock A note on the nonlinear {S}chr\"odinger equation in weak {$L^p$}
  spaces.
\newblock {\em Commun. Contemp. Math.}, 3(1):153--162, 2001.

\bibitem{ChMa-85}
S.-Y.~A. Chang and D.~E. Marshall.
\newblock On a sharp inequality concerning the {D}irichlet integral.
\newblock {\em Amer. J. Math.}, 107(5):1015--1033, 1985.

\bibitem{ChOz-09}
Yonggeun Cho and Tohru Ozawa.
\newblock Sobolev inequalities with symmetry.
\newblock {\em Commun. Contemp. Math.}, 11(3):355--365, 2009.

\bibitem{OMeSe-08}
Jo\~ao~Marcos do~\'O, Everaldo Medeiros, and Uberlandio Severo.
\newblock A nonhomogeneous elliptic problem involving critical growth in
  dimension two.
\newblock {\em J. Math. Anal. Appl.}, 345(1):286--304, 2008.

%\bibitem{DoEsLo-16}
%Jean Dolbeault, Maria J. Esteban, and Michael Loss.
%\newblock Rigidity versus symmetry breaking via nonlinear flows on cylinders and Euclidean spaces.
 %\newblock {\em Invent. Math.}, 206(2):397--440, 2016.
    
\bibitem{FiIlPa-02}
Gadi Fibich, Boaz Ilan, and George Papanicolaou.
\newblock Self-focusing with fourth-order dispersion.
\newblock {\em SIAM J. Appl. Math.}, 62(4):1437--1462, 2002.

\bibitem{Frank2018}
Rupert Frank.
\newblock Private communication.

\bibitem{FrLe-13}
Rupert~L. Frank and Enno Lenzmann.
\newblock Uniqueness of non-linear ground states for fractional {L}aplacians in
  {$\Bbb{R}$}.
\newblock {\em Acta Math.}, 210(2):261--318, 2013.

\bibitem{FrLeSi-16}
Rupert~L. Frank, Enno Lenzmann, and Luis Silvestre.
\newblock Uniqueness of radial solutions for the fractional {L}aplacian.
\newblock {\em Comm. Pure Appl. Math.}, 69(9):1671--1726, 2016.

\bibitem{Frank2016}
Rupert~L. Frank, Elliott~H. Lieb, and Julien Sabin.
\newblock Maximizers for the stein--tomas inequality.
\newblock {\em Geometric and Functional Analysis}, 26(4):1095--1134, Jul 2016.

\bibitem{GiNiNi-79}
B.~Gidas, Wei~Ming Ni, and L.~Nirenberg.
\newblock Symmetry and related properties via the maximum principle.
\newblock {\em Comm. Math. Phys.}, 68(3):209--243, 1979.

\bibitem{GiNiNi-81}
B.~Gidas, Wei~Ming Ni, and L.~Nirenberg.
\newblock Symmetry of positive solutions of nonlinear elliptic equations in
  {${\bf R}^{n}$}.
\newblock In {\em Mathematical analysis and applications, {P}art {A}}, volume~7
  of {\em Adv. in Math. Suppl. Stud.}, pages 369--402. Academic Press, New
  York-London, 1981.

\bibitem{HaLi-35}
G.~H. Hardy and J.~E. Littlewood.
\newblock Notes on the theory of series (xix): A problem concerning majorants
  of fourier series.
\newblock {\em The Quarterly Journal of Mathematics}, os-6(1):304--315, 1935.

\bibitem{Hi-76}
Keijo Hild\'en.
\newblock Symmetrization of functions in {S}obolev spaces and the isoperimetric
  inequality.
\newblock {\em Manuscripta Math.}, 18(3):215--235, 1976.

\bibitem{HiLeSo-18}
Lars Bugiera, Dominik Himmelsbach, Enno Lenzmann, and J{\'e}r{\'e}my Sok.
\newblock Symmetry of boosted ground states for {NLS} with general dispersion.
\newblock Work in preparation.

\bibitem{Ho-83}
Lars H\"ormander.
\newblock {\em The analysis of linear partial differential operators. {II}},
  volume 257 of {\em Grundlehren der Mathematischen Wissenschaften [Fundamental
  Principles of Mathematical Sciences]}.
\newblock Springer-Verlag, Berlin, 1983.
\newblock Differential operators with constant coefficients.

\bibitem{Is-11}
Michinori Ishiwata.
\newblock Existence and nonexistence of maximizers for variational problems
  associated with {T}rudinger-{M}oser type inequalities in {$\Bbb R^N$}.
\newblock {\em Math. Ann.}, 351(4):781--804, 2011.

\bibitem{Kw-89}
Man~Kam Kwong.
\newblock Uniqueness of positive solutions of {$\Delta u-u+u^p=0$} in {${\bf
  R}^n$}.
\newblock {\em Arch. Rational Mech. Anal.}, 105(3):243--266, 1989.

\bibitem{LaLu-13}
Nguyen Lam and Guozhen Lu.
\newblock A new approach to sharp {M}oser-{T}rudinger and {A}dams type
  inequalities: a rearrangement-free argument.
\newblock {\em J. Differential Equations}, 255(3):298--325, 2013.

\bibitem{LiRu-08}
Yuxiang Li and Bernhard Ruf.
\newblock A sharp {T}rudinger-{M}oser type inequality for unbounded domains in
  {$\Bbb R^n$}.
\newblock {\em Indiana Univ. Math. J.}, 57(1):451--480, 2008.

\bibitem{Li-77}
Elliott~H. Lieb.
\newblock Existence and uniqueness of the minimizing solution of {C}hoquard's
  nonlinear equation.
\newblock {\em Studies in Appl. Math.}, 57(2):93--105, 1976/77.

\bibitem{LiLo-01}
Elliott~H. Lieb and Michael Loss.
\newblock {\em Analysis}, volume~14 of {\em Graduate Studies in Mathematics}.
\newblock American Mathematical Society, Providence, RI, second edition, 2001.

\bibitem{Li-60}
J.~E. Littlewood.
\newblock On the inequalities between functions {$f$} and {$f^{\ast} $}.
\newblock {\em J. London Math. Soc.}, 35:352--365, 1960.

\bibitem{MoSc-09}
G.~Mockenhaupt and W.~Schlag.
\newblock On the {H}ardy-{L}ittlewood majorant problem for random sets.
\newblock {\em J. Funct. Anal.}, 256(4):1189--1237, 2009.

\bibitem{Montgomery76}
Hugh~L. Montgomery.
\newblock A note on rearrangements of {F}ourier coefficients.
\newblock {\em Ann. Inst. Fourier (Grenoble)}, 26(2):v, 29--34, 1976.

\bibitem{MoVs-17}
Vitaly Moroz and Jean Van~Schaftingen.
\newblock A guide to the {C}hoquard equation.
\newblock {\em J. Fixed Point Theory Appl.}, 19(1):773--813, 2017.

\bibitem{ReSi-72}
Michael Reed and Barry Simon.
\newblock {\em Methods of modern mathematical physics. {I}. {F}unctional
  analysis}.
\newblock Academic Press, New York-London, 1972.

\bibitem{Ru-05}
Bernhard Ruf.
\newblock A sharp {T}rudinger-{M}oser type inequality for unbounded domains in
  {$\Bbb R^2$}.
\newblock {\em J. Funct. Anal.}, 219(2):340--367, 2005.

\bibitem{RuSa-13}
Bernhard Ruf and Federica Sani.
\newblock Sharp {A}dams-type inequalities in {$\Bbb{R}^n$}.
\newblock {\em Trans. Amer. Math. Soc.}, 365(2):645--670, 2013.

\bibitem{Sp-74}
Emanuel Sperner, Jr.
\newblock Symmetrisierung f\"ur {F}unktionen mehrerer reeler {V}ariablen.
\newblock {\em Manuscripta Math.}, 11:159--170, 1974.

\bibitem{Ta-76}
Giorgio Talenti.
\newblock Best constant in {S}obolev inequality.
\newblock {\em Ann. Mat. Pura Appl. (4)}, 110:353--372, 1976.

\bibitem{We-83}
Michael~I. Weinstein.
\newblock Nonlinear {S}chr\"odinger equations and sharp interpolation
  estimates.
\newblock {\em Comm. Math. Phys.}, 87(4):567--576, 1982/83.

\bibitem{We-87}
Michael~I. Weinstein.
\newblock Existence and dynamic stability of solitary wave solutions of
  equations arising in long wave propagation.
\newblock {\em Comm. Partial Differential Equations}, 12(10):1133--1173, 1987.

\bibitem{We-87b}
Michael~I. Weinstein.
\newblock Solitary waves of nonlinear dispersive evolution equations with
  critical power nonlinearities.
\newblock {\em J. Differential Equations}, 69(2):192--203, 1987.

\bibitem{Zygmund_Trigo}
A.~Zygmund.
\newblock {\em Trigonometric series. {V}ol. {I}, {II}}.
\newblock Cambridge Mathematical Library. Cambridge University Press,
  Cambridge, third edition, 2002.
\newblock With a foreword by Robert A. Fefferman.

\end{thebibliography}

\end{document}